%% file: 08july2025.tex
\documentclass[oneside,reqno,11pt]{amsart}
\usepackage{epsfig}
\usepackage{color}
\usepackage{amssymb,amsmath,amsthm,amstext,amsfonts}
\usepackage[normalem]{ulem}
\usepackage{amsmath,amscd,amstext,amsthm,amsfonts,latexsym}
\usepackage[colorlinks=true,linkcolor=blue, urlcolor=black, citecolor=blue]{hyperref}
\hypersetup{
 citecolor=blue}
\theoremstyle{Theorem A}
\theoremstyle{Theorem B}
\theoremstyle{Theorem C}
\theoremstyle{Theorem D}
\theoremstyle{Theorem E}

\makeatletter

\theoremstyle{definition}
\newtheorem{maintheorem}{Theorem}

\newtheorem{maincorollary}[maintheorem]{Corollary}

\numberwithin{equation}{section}
\numberwithin{figure}{section}
\theoremstyle{plain}
\newtheorem*{cor*}{\protect\corollaryname}
\theoremstyle{plain}
\newtheorem{thm}{\protect\theoremname}[section]
\theoremstyle{definition}
\newtheorem{defn}[thm]{\protect\definitionname}
\theoremstyle{question}

\theoremstyle{remark}
\newtheorem{rem}[thm]{\protect\remarkname}
\theoremstyle{plain}
\newtheorem{prop}[thm]{\protect\propositionname}
\theoremstyle{plain}
\newtheorem{lem}[thm]{\protect\lemmaname}
\theoremstyle{plain}
\newtheorem{cor}[thm]{\protect\corollaryname}

\newtheorem{Conjecture}[thm]{Conjecture}

\input{general-preamble}

\usepackage{float}


\subjclass[2000]{}

\def\Si{\Sigma}

\def\R{\mathbb{R}}
\def\P{\mathbb{P}\R^d}

\def\loc{\text{loc}}

\def\ep{\varepsilon}
\def\vep{\varepsilon}
\def\A{\mathcal{A}}

\def\E{\mathcal{E}}
\def\L{\mathcal{L}}
\def\M{\mathcal{M}_{\text{inv}}}

\def\glr{\text{GL}(d,\R)}

\makeatother

  \providecommand{\corollaryname}{Corollary}
  \providecommand{\definitionname}{Definition}
  \providecommand{\lemmaname}{Lemma}
  \providecommand{\propositionname}{Proposition}
  \providecommand{\remarkname}{Remark}
  \providecommand{\theoremname}{Theorem}
\providecommand{\theoremname}{Theorem}

\usepackage{tikz,xcolor,hyperref}

\definecolor{lime}{HTML}{A6CE39}
\DeclareRobustCommand{\orcidicon}{
	\begin{tikzpicture}
	\draw[lime, fill=lime] (0,0) 
	circle [radius=0.16] 
	node[white] {{\fontfamily{qag}\selectfont \tiny ID}};
	\draw[white, fill=white] (-0.0625,0.095) 
	circle [radius=0.007];
	\end{tikzpicture}
	\hspace{-2mm}
}
\foreach \x in {A, ..., Z}{\expandafter\xdef\csname orcid\x\endcsname{\noexpand\href{https://orcid.org/\csname orcidauthor\x\endcsname}
			{\noexpand\orcidicon}}
}

\renewcommand{\le}{\leqslant}
\renewcommand{\leq}{\leqslant}
\renewcommand{\geq}{\geqslant}
\renewcommand{\ge}{\geqslant}

\author{Reza Mohammadpour \orcidA{} and Paulo Varandas
\orcidB{}}

\address{Reza Mohammadpour, Department of Mathematics, Uppsala University, Box 480, SE-75106, Uppsala, Sweden.}
\email{reza.mohammadpour@math.uu.se}

\address{Paulo Varandas, Departamento de Matemática
da Universidade de Aveiro
Campus Universitário de Santiago, Portugal.}
\email{paulo.varandas@ua.pt}

\date{\today}


\begin{document}

\title[Statistical properties of equilibrium states for fiber-bunched cocycles]{Statistical properties of equilibrium states for fiber-bunched matrix cocycles and applications}

\maketitle

\begin{abstract}
We contribute to the thermodynamic formalism of H\"older continuous fiber-bunched matrix cocycles, Anosov diffeomorphisms, and hyperbolic repellers. Specifically, we prove that $1$-typical fiber-bunched cocycles $\mathcal{A}$ over topologically mixing subshifts of finite type admit a unique Gibbs equilibrium state $\mu_t$ associated with the non-additive family of potentials $\{t \log \|\mathcal{A}^n\|\}_{n \in \mathbb{N}}$, for a range of parameters $t \in (-t_*, +\infty)$, where $t_* > 0$. Furthermore, these equilibrium states are $\psi$-mixing, therefore weak Bernoulli. In addition, these results allow us to derive consequences for the thermodynamic formalism of open sets of hyperbolic repellers and Anosov diffeomorphisms. In particular, it provides a positive answer to a conjecture posed by Gatzouras and Peres \cite{Gatzouras1997} for $C^1$-open sets of $\alpha$-fiber-bunched hyperbolic repellers.

\end{abstract}


\color{black}
\section{Introduction}

\color{black}

The study of equilibrium states for matrix cocycles has seen significant advances in recent years, particularly through the lens of thermodynamic formalism.  Let $(\Sigma, \sigma)$ be a topologically mixing subshift of finite type and let $\M(\sigma)$ denote the space of $\sigma$-invariant probability measures. Given a matrix cocycle $\A:\Sigma \to \glr$ over $(\Sigma,\sigma)$ and $t\in \R$, we consider the family of potentials $t\Phi_{\A}:=\{t\log \|\A^n\|\}_{n \in \N}$, where
$$
\mathcal A^n(x) = \A(\sigma^{n-1}(x)) \cdot \dots \cdot \A(\sigma(x))  \cdot \A(x)
$$
for each $n\ge 1$ and $x\in \Sigma$ and denote the top Lyapunov exponent of a $\sigma$-invariant probability measure $\mu$ by 
$$
\lambda_1(\mu,\A)=\lim_{n\to\infty} \frac1n \int \log \|\A^n(x)\|\, d\mu(x).
$$
By \cite{CFH, CPZ}, the topological pressure of the potential $t\Phi_{\A}$ is given by
\[ P(\sigma, t\Phi_{\A})=\sup_{\eta\in  \M(\sigma)} 
\Big\{
h_\eta(\sigma)+t \,\lambda_1(\eta,\A)
\Big\},\]
where $h_\eta(\sigma)$ is the measure theoretical entropy.
The family of potentials $t\Phi_{\A}$ is a well-known to be sub-additive 
whenever $t\geq 0$, meaning that
$ 
t\log \|\A^{m+n}(x)\|\le t\log \|\A^{n}(x)\|+ t\log \|\A^{m}(\sigma^n(x))\|
$ 
for every $x\in \Sigma$ and $m,n\in \N$. In this context, there are always equilibrium states, that is, probability measures $\mu$ such that 
$$
h_\mu(\sigma)+t \,\lambda_1(\mu,\A)  = P(\sigma, t\Phi_{\A})
$$
as a consequence of the upper semi-continuity of both the entropy map $\mu\mapsto h_{\mu}(\sigma)$ and $\mu\mapsto\lambda_1(\mu,\A)$ and 
weak$^*$-compactness of $\M(\sigma)$
(see e.g. \cite{Bar, CFH}).
For $t\geq 0$, a key concept in the theory is the quasi-multiplicativity property. 
Feng \cite{Feng11} established that this property implies the existence of a unique Gibbs measure $\mu_t$ for the potential $t\Phi_{\mathcal{A}}$ when $t \geq 0$. Park \cite{park2020quasi} showed that 1-typical cocycles satisfy quasi-multiplicativity, making them a natural class for which Gibbs measures exist for non-negative parameters.
Apart from existence and uniqueness, the statistical properties of such equilibrium states are not yet fully understood. Some statistical properties, including large deviation principles and the central limit theorem, have been established (see \cite{GouezelStoyanov2019,Park-Piraino} and references therein).  Moreover, for one-step cocycles, Piraino~\cite{Piraino20} proved that all such equilibrium states are weak Bernoulli under strong irreducibility and proximality assumptions, while Morris~\cite{Morris21} showed that if an equilibrium state is totally ergodic, then it is $\psi$-mixing and therefore weak Bernoulli. Recently, the first-named author and Park~\cite{MP-uniform-qm} showed that the equilibrium state is totally ergodic, $\psi$-mixing, and hence weak Bernoulli, for one-step cocycles under the irreducibility assumption.

\smallskip
The behavior of the family of potentials $t\Phi_\A$ for $t < 0$ diverges significantly and is much less understood. In fact, if $t<0$ the family $t\Phi_\A$ is super-additive, meaning that  
$$ 
t\log \|\A^{m+n}(x)\|\ge t\log \|\A^{n}(x)\|+ t\log \|\A^{m}(\sigma^n(x))\|
 $$
for every $x\in \Sigma$ and $m,n\in \N$, and equilibrium states may fail to exist. Beyond the simpler cases of diagonal matrices, positive matrices, and dominated cocycles,
where equilibrium states are known to exist and are Gibbs measures, few results are known. One notable exception is the work of Rush \cite{Rush} for one-step cocycles, who employed perturbation techniques and the transfer operator developed in \cite{Park-Piraino} to prove the existence of a unique Gibbs measure ${\mu_t}$ for the potential $t\Phi_{\mathcal{A}}$ when $t < 0$ is sufficiently close to zero, under the mild assumption of 1-typicality.
Quite recently, the first-named author and Quas \cite{mohammadpour-Quas} studied phase transitions for one-step cocycles and proved that there exists a strongly irreducible and proximal one-step cocycle $\A$ and $t<0$ so that the potential $t\Phi_{\A}$ has more than one equilibrium state. This result emphasizes that phase transitions for negative $t$ can appear even in the case of one-step cocycles.  
No results seem to have been proved concerning the thermodynamic formalism of non-locally constant H\"older continuous cocycles for negative parameters $t$.

\smallskip
In this paper, we aim at contributing to the thermodynamic formalism of linear cocycles over a subshift of finite type $(\Sigma,\sigma)$, beyond the context of one-step cocycles. We will always consider fiber-bunched linear cocycles $\A$, which roughly means that the projective action is dominated by the hyperbolicity constants of the shift map
(alternatively a fiber-bunched cocycle is one such that its projectivized skew products $F_{\mathcal{A}}: \Sigma \times \P \rightarrow \Sigma \times \P$ defined by ~\eqref{skew-product} are partially hyperbolic).
We shall assume that the fiber-bunched  H\"older continuous cocycle satisfies the generic assumptions (pinching and twisting) introduced by Bonatti and Viana \cite{bonatti2004lyapunov} and further used by Avila and Viana \cite{avila2007simplicity}. The space of these cocycles, referred to as $1$-typical (cf. Subsection \ref{subsec:typical} for the precise definitions), is a $C^0$-open and dense subset of the space of fiber-bunched H\"older continuous cocycles.

\smallskip
The first results concern the existence and uniqueness of equilibrium states, as well as the regularity of Lyapunov exponents. We prove that there exists $t_* > 0$ such that, for every $t \in (-t_*, +\infty)$, the potential $t\Phi_{\A}$ has a unique equilibrium state $\mu_t$, which is a Gibbs measure, and whose top Lyapunov exponent depends analytically on the parameter $t$ in an open interval containing the origin (cf. Theorem~\ref{thm:Main1}).
We also prove similar results on the analyticity for the sum of $k$ largest Lyapunov exponents for each $1\le k \le d$ in the case of typical cocycles (cf. Corollary~\ref{cor:Main1}). 
The equilibrium states in Theorem~\ref{thm:Main1} are obtained as the natural extension of one-sided shift invariant probability measures $\mu_t$, defined using leading eigenvalues and leading eigenvectors of suitable transfer operators $\L_t$ given by 
\begin{equation*}
\mathcal{L}_t f(x, {u}):=\sum_{y \in \sigma^{-1}x} g(y)\left\|\mathcal{A}(y)^* \frac{u}{\|u\|}\right\|^t f\left(y, \overline{\mathcal{A}(y)^* u}\right),
\end{equation*}
 and acting on the space of real valued H\"older continuous functions $f$ on $\Sigma \times \P$ (we refer the reader to Sections~\ref{sec: Transfer} and ~\ref{sec:Gibbs-eq-states} for details).
This extension of the results by Rush \cite{Rush} to the broader setting of fiber-bunched cocycles introduces significant new challenges. In fact, as in \cite{Rush}, the author deals with one-step cocycles, the associated transfer operator—similar to the classical one introduced in \cite{lepage1982theoremes}—depends only on a finite set of matrices. Consequently, the stationary measure $\nu_t$ (the leading eigenmeasure of $\mathcal{L}_t$) depends solely on $\P$, an important fact that plays a key role in the proof that ${\mu}_t$ satisfies the Gibbs property.
Unlike one-step cocycles, which form a basic class of matrix cocycles, fiber-bunched matrix cocycles constitute a natural and significantly richer class.
In this context, the stationary measures $\nu_t$ still exist but depend on the entire product space $\Sigma \times \P$.  To handle this additional complexity, we follow a different approach, such as the one used by Bonatti and Viana \cite{bonatti2004lyapunov}. 

\smallskip
Subsequently, we study the statistical properties of equilibrium states in greater detail.
In fact, we prove that the equilibrium state ${\mu}_t$ is $\psi$-mixing and hence weak Bernoulli for each $t \in (-t_{\ast}, \infty)$ (we refer the reader to Theorem~\ref{thm:Main2} for the precise statements).  Note that $\psi$-mixing is stronger than usual mixing (see \cite{Bradley05}). Our result extends the work of Call and Park \cite{Call2023}, who proved that the equilibrium state ${\mu}_t$ is weak Bernoulli for $t \geq 0$.

\smallskip
The previous findings also carry substantial implications for smooth dynamical systems $T$ on a compact manifold $M$, as we now describe.
Due to the existence of finite Markov partitions, a uniformly hyperbolic map is semiconjugate to a subshift of finite type. Using the H\"older continuity of the semiconjugacies one can identify the derivative cocycle $DT\mid_F$ over a $DT$-invariant subbundle $F\subset TM$ (see e.g. \cite{Mohammadpour-Varandas}).
This leads to the natural question about the thermodynamic formalism for the family of potentials $t\log \|DT^{n}\mid_{F}\|$ when $t \in \R$. 
In the case where $DT\mid_F$ is conformal, the problem reduces to studying equilibrium states for Hölder continuous potentials, whose theoretical foundations are well established (see e.g. \cite{baladi2000positive, Bow, dolgopyat1998decay, liverani1995decay, viana2016foundations, ruelle2004thermodynamic}). 
In the general (non-conformal) case, while the family $t\log \|DT^{n}\mid_F\|$ is non-additive, one can apply the results described above for linear cocycles. These results establish the existence of Gibbs-type states and their mixing properties for certain non-additive potentials, extending beyond the strongly conformal, reducible, and dominated settings.
This has important applications for hyperbolic repellers and Anosov diffeomorphisms, as we now discuss.   
There is a well-known conjecture by D. Gatzouras and Y. Peres \cite{Gatzouras1997}:

\smallskip
\begin{Conjecture}\label{conjecture} \emph{Let $T: M \rightarrow M$ be an expanding map, and let $K \subseteq M$ be a compact invariant set that satisfies the specification property. Then there exists a unique ergodic  $T$-invariant
measure with the same Hausdorff dimension as $K$. Moreover, it is mixing for $T$ and, possibly, measurably isomorphic to a Bernoulli shift.}
\end{Conjecture}
\smallskip
Although this conjecture was ultimately disproven—since nonconformal repellers may fail to support measures of maximal dimension (see \cite{BarralFeng2011})—it motivated the investigation of the mixing properties of such measures, which, whenever they do exist, are measurably isomorphic to a Bernoulli shift. Let us discuss some of its motivations. It is long known that if $T: M \to M$ is a conformal, transitive, $C^{1+\alpha}$ expanding map, then its absolutely continuous invariant measure is the unique equilibrium state for the potential $-\log|\operatorname{det}DT(x)|$. 
For this reason, Conjecture~\ref{conjecture} has already been well understood in the special case of repellers for conformal expanding maps, in which case all Lyapunov exponents of an invariant measure coincide. In contrast, the nonconformal expanding case presents significantly more challenges, as the relevant potentials may be non-additive and may depend on multiple distinct Lyapunov exponents. 
In order to understand the measures of maximal dimension in this setting, it appears essential to analyze the equilibrium states for potentials of the form $t \log \|DT(x)\|$, for $t\in \mathbb R$.
In \cite{Morris21}, Morris showed that totally ergodic equilibrium measures for one-step cocycles are $\psi$-mixing and weak Bernoulli. By a result of Friedman and Ornstein \cite{Friedman70}, this gives an affirmative answer to a special case of the conjecture, namely for self-affine repelling sets with generic translations where the map $DT(x)$ is locally constant. 
In this paper, we consider a more general class of repellors, which persist under $C^1$-perturbations of the dynamics. We address the longstanding conjecture of Gatzouras and Peres (see Theorem~\ref{main:t2}) and significantly improve upon the results of \cite{Morris21}.

Similar questions can be posed concerning the regularity and mixing properties of equilibrium states for the top Lyapunov exponent of Anosov diffeomorphisms. This does not follows from the classical theory as there exist $C^1$-open sets of Anosov diffeomorphisms $T$ with a hyperbolic splitting $TM=E^s\oplus E^u$ so that the cocycle $DT\mid_{E^u}$ contains both elliptic matrices and matrices with a dominated splitting (cf. \cite{BRV18} for more details), in which case $t \log \|DT^n(x)\mid_{E^u}\|$ is sub-additive for $t \geq 0$ and super-additive for $t<0$. As a byproduct of our results for linear cocycles we describe the thermodynamic formalism of open classes of Anosov diffeomorphisms, which satisfy a bunching condition along the unstable subbundle
(see Theorem~\ref{main:t1}).

\smallskip
This paper is organized as follows. In Section~\ref{sec:statements} we present the framework of the paper and state the main results.  Section~\ref{prel} is devoted to some preliminaries on subshift dynamics, basic notions of total ergodicity and weak Bernoulicity of measures and metrics of projective spaces. The fiber-bunching, existence of invariant holonomies and typicality of linear cocycles is discussed in detail in Section~\ref{subsecmatrixc}, where we collect some results from \cite{bonatti2004lyapunov}. In Sections~\ref{sec: Transfer} and ~\ref{sec:Gibbs-eq-states} contain the core results of the paper, including the results on the disintegration of stationary measures and uniqueness of Gibbs equilibrium states.  In Section~\ref{Sec: psi mixing} we prove the results on weak Bernoulicity and $\psi$-mixing. Finally, the applications for hyperbolic repellers and Anosov diffeomorphisms are given in Sections~\ref{proof of TheoremC}.

\section{Statement of the main results}\label{sec:statements}

Through this paper, we will always assume that 
$\sigma:\Sigma \to \Sigma$ is a topologically mixing subshift of finite type,
where $\Sigma \subset \{1,2,\dots,k\}^{\Z}$ is the symbolic space, endowed with the metric $d$ defined in ~\eqref{metric}, and $\sigma(x_{n})_{n\in \Z}=(x_{n+1})_{n\in \Z}$ for any $(x_{n})_{n\in \Z} \in \Sigma$.

\medskip

Let $\A: \Sigma \to \glr$ be a matrix cocycle over $(\Sigma, \sigma)$. For each $\sigma$-invariant probability measure $\mu$ such that $\log \left\|\mathcal{A}^{ \pm 1}\right\| \in L^1(\mu)$, Oseledets' theorem ensures that for $\mu$-almost every $x$ there exist $k(x) \geqslant 1$ and an $\mathcal{A}$-invariant splitting $\mathbb{R}^d=E_x^1 \oplus E_x^2 \oplus \cdots \oplus E_x^{k(x)}$ so
that the Lyapunov exponents are well defined by the limits
$$
\chi_i(\mathcal{A}, x)=\lim _{n \rightarrow \infty} \frac{1}{n} \log \left\|\mathcal{A}^n(x) v\right\|, \quad \text { for every } v \in E_x^i \backslash\{0\}
$$
We denote by $\lambda_1(\mathcal{A}, x) \geqslant \lambda_2(\mathcal{A}, x) \geqslant \ldots \geqslant \lambda_d(\mathcal{A}, x)$ the Lyapunov exponents, counted with multiplicity, of the cocycle $\mathcal{A}$ at $x \in \Sigma$. Set also $\lambda_i(\mu, \mathcal{A}):=\int \lambda_i(\mathcal{A}, \cdot) d \mu$ for $1 \leqslant i \leqslant d$. We note that the Lyapunov exponents $\lambda_i(\mathcal{A}, x)$ can be computed as
$$
\lambda_i(\mathcal{A}, x)=\lim _{n \rightarrow \infty} \frac{1}{n} \log \sigma_i\left(\mathcal{A}^n(x)\right)
$$
where $\sigma_1, \ldots, \sigma_d$ are the singular values, listed in decreasing order according to multiplicity \cite{horn1994topics}. For the simiplisity, we denote $\lambda(\mu, \A):=\lambda_1(\mu, \A).$

\begin{rem}
In the special case of uniformly hyperbolic dynamical systems
(including uniformly expanding maps and Anosov diffeomorphisms), dynamical cocycles can assume a simpler formulation.  Indeed, as these admit a symbolic coding by a subshift of finite type \cite{Bow}, the derivative cocycle of a uniformly hyperbolic map can effectively be regarded as a linear cocycle over a subshift of finite type. We refer the reader to Section~\ref{proof of TheoremC} for more details.    
\end{rem}

\medskip
\subsection{Main results} As mentioned above, in this paper we consider H\"older continuous fiber-bunched linear cocycles over the topologically mixing subshift of finite type $(\Sigma, \sigma)$.

\begin{maintheorem}
    \label{thm:Main1}
    \emph{Let $\mathcal{A}: \Sigma \rightarrow GL(d, \R)$ be a fiber-bunched linear cocycle over $(\Sigma, \sigma)$. Assume $\A$ is $1$-typical. There exists $t_*>0$ such that:
    \begin{enumerate}
        \item  For each $t\in (-t_*,+\infty)$, there exists a unique equilibrium state $\mu_t$ for $t \Phi_{\A}$ and, moreover, $\mu_t$ is a Gibbs measure;  
        \item $(-t_*,t_*)\ni t \mapsto P(\sigma, t\Phi_{\A})'= \lambda_1({\mu_t}, \A)$ is real analytic.
    \end{enumerate}
    }
\end{maintheorem}



Recently, the first-named author and Quas~\cite{mohammadpour-Quas} constructed an example of a matrix cocycle satisfying our assumptions, for which the equilibrium state for $t\Phi_{\mathcal{A}}$ is not unique for some $t < 0$. This demonstrates that one should not expect the above result to hold for all $t < 0$.

\medskip
It is worth noticing that if $\A$ is a $GL(d, \R)$ typical cocycle then $\A^{\wedge k}$ is a $ GL\bigg({\small\big(\!\begin{array}{c} d \\ k \end{array}\!\big)},\mathbb R \bigg)$ $1$-typical cocycle and, given any $\sigma$-invariant probability measure $\eta$ one has 
$$
\lambda_1(\eta, \A^{\wedge k})=\sum_{i=1}^k
\lambda_i(\eta, \A).
$$ 
Together with Theorem~\ref{thm:Main1}, this yields the following direct consequence:

\begin{maincorollary}
    \label{cor:Main1}
    \emph{Let $\mathcal{A}: \Sigma \rightarrow GL(d, \R)$ be a $\A$ is typical fiber-bunched linear cocycle over $(\Sigma, \sigma)$. There exists $t_*>0$ such that, for each $1\le k \le d$, the following hold:
        \begin{enumerate}
        \item  For each $t\in (-t_*,+\infty)$, there exists a unique equilibrium state $\mu^{k}_t$ for $t \Phi_{\A^{\wedge k}}$ and, moreover, $\mu^k_t$ is a Gibbs measure;  
        \item For every $1\le i \le d$, the Lyapunov exponent functions
    $$
    (-t_*,t_*)\ni t \mapsto \sum_{i=1}^k
\lambda_i(\mu_t^k, \A)
    $$
    are real analytic with respect to $t$.
    \end{enumerate}
    }
\end{maincorollary}

The previous discussion prompts a question about the regularity of all Lyapunov exponents with respect to a fixed family of equilibrium states. More precisely, it is natural to ask the following:

\medskip
\noindent{\bf Question:}
 Under which conditions can one prove that all Lyapunov exponents functions $
    (-t_*,t_*)\ni t \mapsto 
\lambda_i(\mu_t^1, \A)
    $     
to be real analytic? 
\medskip

\medskip
Other related questions concerns the finer ergodic properties of the equilibrium states, in particular the Bernoulli property and $\psi$-mixing (see Subsection \ref{sebsec:weakly Bern} for the definition). 
We say that an invariant measure $\mu \in \M(\sigma)$ is {\em quasi-Bernoulli} if there is $D>0$ such that for any $I, J \in \Sigma^*$, we have
\begin{equation}\label{QB}
    D^{-1}\mu([I])\mu([J])\leq \mu([IJ]) \leq D \mu([I]) \mu([J]).
\end{equation}
It is clear that Gibbs measures for additive or almost additive sequences of potentials are quasi-Bernoulli. 


\begin{maintheorem}
\label{thm:Main2}
    \emph{Let $\mathcal{A}: \Sigma \rightarrow GL(d, \R)$ be a fiber-bunched linear cocycle over $(\Sigma, \sigma)$, and let $t_* > 0$ and $\mu_t$ be as given by Theorem~\ref{thm:Main1}.
For each $t\in (-t_*,+\infty)$, the equilibrium measure $\mu_t$ is totally ergodic and $\psi$-mixing:$$\lim_{n\to \infty}\sup_{I, J \in \Sigma^{\ast}}\left|\frac{\mu_{t}\left([I] \cap \sigma^{-n-|I|}[J]\right)}{ \mu_{t}([I])  \mu_{t}([J])}-1\right|=0.
$$
In particular, ${\mu_t}$ is quasi-Bernoulli and weak Bernoulli. 
}
\end{maintheorem}

The proof of Theorem~\ref{thm:Main2} occupies Section \ref{Sec: psi mixing}.

\subsection{Applications to hyperbolic repellers and Anosov diffeomorphisms} 
In this subsection, we derive several applications of the previous main results to the context of smooth dynamical systems. 
Let us recall some notions and introduce the notation.

\smallskip
Let $M$ be a Riemannian manifold and $T$ be a $C^1$-smooth map on $M$. A compact $T$-invariant subset $\Lambda \subset M$ is a \emph{hyperbolic repeller} if
 \begin{enumerate}
 \item[(i)] 
 there exists an adapted norm $\|\cdot\|$ and $\lambda>1$ such that
$$
\left\|DT(x)v\right\| \geq \lambda\|v\|\qquad 
\text{for all $x \in \Lambda$ and $v \in T_x M$,}
$$
\item[(ii)]  there exists a bounded open neighborhood $V$ of $\Lambda$ such that
$$
\Lambda=\Big\{x \in V: T^n x \in V \text { for all } n \geq 0\Big\} .
$$
 \end{enumerate}
 Given such a hyperbolic repeller $\Lambda$ and $\alpha \in(0,1]$, we say that $\left.T\right|_{\Lambda}$ is \emph{$\alpha$-bunched} if
\begin{equation}
    \label{def:alhpabunchedT}
\left\|DT(x)^{-1}\right\|^{1+\alpha} \cdot\left\|DT(x)\right\|<1
\qquad 
\text{for all $x \in \Lambda$.}
\end{equation}
We note that the 1-bunching assumption for repellers was first 
studied by Falconer \cite{falconer94}, and that $C^1$-open classes of $\alpha$-bunched repellers can be obtained by $C^1$-small perturbations of conformal repellers. 
 The next theorem provides a solution to Gatzouras and Peres' Conjecture \ref{conjecture} for open classes of hyperbolic repellers. More precisely:

\begin{maintheorem}\label{main:t2}
 \emph{Let $M$ be a Riemannian manifold, $r \ge 2$ and $T$ be $C^r$-smooth map on $M$.  If $\Lambda \subset M$ is an $\alpha$-bunched repeller for $T$, with $\alpha \in (0,1)$, then there exists a $C^1$-open neighborhood $\mathcal{V}_1\subset C^r(M, M)$ of $T$ and a $C^1$-open and $C^r$-dense subset $\mathcal{V}_2 \subset \mathcal{V}_1$ with the following property:  For every $S \in \mathcal{V}_2$,  there exists $t_*>0$ such that:
\begin{itemize}
    \item[(a)]For every $t \in (-t_*,+\infty)$, there exists a unique equilibrium state ${\mu_t}$ for the family of potentials $\{ t \log \| DS_{|\Lambda_S}^n(x) \| \}_{n\ge 1}$ and it is a Gibbs measure.  
    \item[(b)]For every $t \in (-t_*,+\infty)$:
    \begin{itemize} 
        \item[(b1)]the equilibrium state ${\mu_t}$ is $\psi$-mixing, that is,  
\[
\lim_{n\to\infty} \sup_{\substack{B_1 \in \mathcal{F}_{-\infty}^0,\, B_2 \in \mathcal{F}_n^{\infty} \\ {\mu_t}(B_1) > 0,\, { \mu_t}(B_2) > 0}} \left| \frac{{\mu_t}(B_1 \cap B_2)}{{ \mu_t}(B_1){ \mu_t}(B_2)} - 1 \right|=0,
\]  
where one considers the sigma-algebras $\mathcal{F}_{a}^b=\sigma(S^{-n}(B) \colon B\in \mathcal B, a\le n\le b)$ generated by $S$ between times $a$ and $b$,
\item[(b2)]the equilibrium state ${\mu_t}$ is totally ergodic, quasi-Bernoulli, $\psi$-mixing and weak Bernoulli. In particular, $(\sigma, {\mu_t})$ is conjugate to a Bernoulli shift.
    \end{itemize}
    \item[(c)]$(-t_*,t_*)\ni t \mapsto P(\sigma, t\log\| DS_{|\Lambda_S}^n(x)\|)'= \lambda_1({\mu_t}, D S_{|\Lambda_S})$ is real analytic.
    \end{itemize}
    }
 \end{maintheorem}

Similar results can be derived for $C^1$-open sets of Anosov diffeomorphisms. Let us introduce the notation. 
Recall that a $C^1$-diffeomorphism  $T$ on a compact manifold $M$ is called an \emph{Anosov diffeomorphism} if there exist a $D T$-invariant splitting $T M=E^s \oplus E^u$ and constants $C>0, \zeta \in(0,1)$ such that 
$$
\left\|\left.DT^n(x)\right|_{E_x^s}\right\| \leq C \zeta^n \text { and }\left\|\left.D_x f^{-n}\right|_{E^u_x}\right\| \leq C \zeta^n 
$$
for every $x\in M$ and $n\ge 1$.
 Using that 
the unstable and stable bundles $E^s$ and $E^u$ of a  $C^{1+\alpha}$ Anosov diffeomorphism $T$  are $\beta$-H\"older continuous, for some $\beta \in(0, \alpha]$ (cf. \cite{HP}), it is possible to show that 
 the restrictions of the derivative cocycle of a $C^{1+\alpha}$ Anosov diffeomorphism to the stable and unstable subbundles can be modeled by a H\"older continuous cocycle over a subshift of finite type (cf. \cite{Mohammadpour-Varandas}). 
%
We will say the derivative cocycle $\left.D T\right|_{E^u}$ is \emph{fiber-bunched} if there exists $n\ge 1$ such that
\begin{equation}
\label{eq:bunchingDF}
\left\|\left.DT^n(x)\right|_{E_x^u}\right\| \cdot\left\|\left(\left.DT^n(x)\right|_{E_x^u}\right)^{-1}\right\| \cdot \max \left\{\left\|\left.DT^n(x)\right|_{E_x^s}\right\|^\beta,\left\|\left(\left.DT^n(x)\right|_{E_x^u}\right)^{-1}\right\|^\beta \right\}<1 .
\end{equation}
When $\left.D T\right|_{E^u}$ is fiber-bunched, the canonical stable and unstable holonomies for the cocycle $\left.D T\right|_{E^u}$ converge and are $\beta$-H\"older continuous (see Subsection~\ref{subsec:fiberbunching} for the definition of holonomy maps). 

We obtain the following:

\begin{maintheorem}\label{main:t1}
 \emph{Let $T$ be a transitive $C^{1+\alpha}$ Anosov diffeomorphism of a compact Riemmanian manifold $M$ such that $\left.D T\right|_{E^u}$ is fiber-bunched. If $\left.D T\right|_{E^u}$ is a typical cocycle, then there exists $t_*>0$ such that:
\begin{itemize}
\item[(a)]For every $t\in (-t_{\ast}, \infty)$,
\begin{itemize}
    \item[(a1)]there exists a unique equilibrium state ${\mu_t}$ for the potentials $\{ t \log\|D T^n\|_{E^u} \}_{n\ge 1}$ and it is a Gibbs measure;  
    \item[(a2)]  the equilibrium state ${\mu_t}$ is $\psi$-mixing, totally ergodic, quasi-Bernoulli and weak Bernoulli.
    \end{itemize}
    \item[(b)]$(-t_*,t_*)\ni t \mapsto P(\sigma, t\log\|D T^n|_{E^u}\|)'= \lambda_1({\mu_t}, D T|_{E^u}))$ is real analytic.
    \end{itemize}
}
\end{maintheorem}

As a final comment, we note that similar results can be derived for $DT\mid_{E^s}$ (it is enough to consider the Anosov diffeomorphism $T^{-1}$ and to observe that the unstable subbundle for $T^{-1}$ is the stable subbundle $E^s$ for $T$).

\color{black}

\section{Preliminaries}\label{prel}

In Subsections~\ref{subsecSFT} to \ref{subsecmatrixc} we introduce some notation and recall basic properties on subshifts of finite type, weak Bernoulicity, projective spaces, and linear cocycles. Subsection~\ref{subsec:fiberbunching} addresses on the properties of invariant holonomies for fiber-bunched cocycles,
while the notion of typical cocycles is recalled in Subsection~\ref{subsec:typical}. Finally, in 
Subsection~\ref{Non-additive TF} we recall some necessary results on the non-additive thermodynamic formalism. 
The reader acquainted with these topics may choose to skip this section in a first reading, returning to it whenever necessary.

\subsection{Subshifts of finite type} \label{subsecSFT}

Given a transition matrix $Q \in \mathcal M_{k\times k}(\{0,1\})$ the one-sided
subshift of finite type determined by $Q$ is a  left shift map
$\sigma : \Sigma_{Q}^{+} \to  \Sigma_{Q}^{+} $ defined as $\sigma(x_n)_{n\in \N}$ = $(x_{n+1})_{n\in \N}$, acting on the space of sequences
\[
\Sigma_{Q}^{+}:=\Big\{x=(x_{i})_{i\in \N} : x_{i}\in \{1,...,k\} \hspace{0.2cm}\textrm{and} \hspace{0.1cm}Q_{x_{i}, x_{i+1}}=1 \hspace{0.2cm}\textrm{for all}\hspace{0.1cm}i\in \N\Big\}.
\]
Similarly, we define a two-sided subshift of finite type  $(\Sigma_{Q}, \sigma)$  on the space 
\[
\Sigma_{Q}:=\Big\{x=(x_{i})_{i\in \Z} : x_{i}\in \{1,...,k\} \hspace{0.2cm}\textrm{and} \hspace{0.1cm}Q_{x_{i}, x_{i+1}}=1 \hspace{0.2cm}\textrm{for all}\hspace{0.1cm}i\in \Z\Big\}.
\]
For notational simplicity, we shall omit the dependence of $Q$ on the space $\Sigma_Q$ and $\Sigma_{Q}^{+}$ and always write $\Sigma$ and $\Sigma^{+}$, respectively. Letting $\Pi: \Sigma \rightarrow \Sigma^{+}$ denote the standard projection, we have $\Pi x=\hat{x}$ where $x=\left(x_i\right)_{i \in \mathbb{Z}}$ and $\hat{x}=\left(x_i\right)_{i \in \mathbb{N}_0}$.  

\medskip
 For any  $x \in \Sigma$ and $n\ge 1$, we define $n$-cylinder defined by $x$ as
$$[x]_n:=\left\{\left(y_i\right)_{i \in \mathbb{Z}} \in \Sigma: x_i=y_i \text { for all } 0 \leq i \leq n-1\right\} .$$

We denote by $\Sigma_n$ the set of all admissible words of length $n$ of $\Sigma$ and write $\Sigma^*=\bigcup_{n\ge 1} \Sigma_n$. For each $I\in \Sigma$ let $|I|$ denote the length of the word $I$ and let
$$
[I]=\bigg\{x\in \Sigma \colon x_j=i_j \; \text{for every } \, 0\le j \le |I|-1\bigg\}
$$
be the cyclinder set in $\Sigma_Q$ determined by $I$. The concatenation of two words $i \in \Sigma^* \cup \Sigma$ and
$j \in \Sigma*$ is denoted by $ij$.

Assume that $Q\in \mathcal M_{k\times k}(\{0,1\})$
is a transition matrix and $(\Sigma_Q, \sigma)$ is the corresponding subshift of finite type.
It is well known that the primitivity of $Q$ (i.e., the existence of an integer $n\ge 1$ such that all the entries of $Q^n$ are positive) is equivalent to
the property of the subshift of finite type $(\Sigma_Q, \sigma)$ to be \textit{topologically mixing}.
%
%
%
%
%
Endow the space $\Sigma=\Sigma_Q$ with the following metric $d$: for $x=(x_{i})_{i\in \Z}, y=(y_{i})_{i\in \Z} \in \Sigma$
\begin{equation}\label{metric}
d(x,y):= 2^{-\inf{\{k\ge 0} \colon \text{$x_{i}\neq y_{i}$ for some $|i| \le k$} \}}.
\end{equation} 
For any $\hat{x} \in \Sigma^{+}$, we likewise define the cylinder $[\hat{x}]_n$ as $\Pi\left([x]_n\right)$ for any $x \in \Pi^{-1}(x)$. Similarly, we equip the one-sided subshift $\Sigma^{+}$ with the same metric. 

\smallskip
The \textit{local stable set} at $x=(x_{i})_{i\in \Z}$ is the set
$
W_{\loc}^{s}(x)=\{(y_{n})_{n\in \Z} : y_{n}=x_{n} \hspace{0,2cm}\textrm{for all}\hspace{0.2cm} n\geq 0\} 
$
while the \textit{local unstable set}
at $x=(x_{i})_{i\in \Z}$ is defined by
$ W_{\loc}^{u}(x)=\{(y_{n})_{n\in \Z} : y_{n}=x_{n} \hspace{0,2cm}\textrm{for all}\hspace{0.2cm} n \leq 0\} .$ 
The global stable (resp. global unstable) manifold of $x \in \Sigma$ is 
\[W^{s}(x):=\left\{y \in \Sigma: \sigma^{n} y \in  W_{\loc}^{s}(\sigma^{n}(x))\text { for some } n \geq 0\right\},\]
\[
(\text{resp.} \quad W^{u}(x):=\left\{y \in \Sigma: \sigma^{n} y \in W_{\loc}^{u}(\sigma^{n}(x)) \text { for some } n \leq 0\right\}).
\]
The two-side subshift of finite type  $(\Sigma, \sigma)$ 
equipped with the metric $d$ in \eqref{metric}
is a hyperbolic homeomorphism (see \cite[Subsection 2.3]{AV10} for more details) and it has a local product structure defined by
\begin{equation}
    \label{eq:prodst}
[x, y]:=W_{\loc}^{u}(x) \cap W_{\loc}^{s}(y)
\end{equation}
for any $x=(x_{i})_{i\in \Z}, y=(y_{i})_{i\in \Z} \in \Sigma$ so that $x_{0}=y_{0}$. 

\subsection{Totally ergodic and weak Bernoulli measures}\label{sebsec:weakly Bern}
We denote by $\E(\sigma)$ the space of all ergodic measures
on $\Si$. Moreover, we say that an invariant measure $\mu \in \M(\sigma)$ is \textit{totally ergodic} if $\mu$ is ergodic with respect to $\sigma^n$, for every $n \in \N$.

Let $\mu \in \M(\Sigma)$. Two partitions $\mathcal{P}$ and $Q$ are called {\em $\varepsilon$-independent} if
$$
\sum_{P \in \mathcal{P}, Q \in \mathcal{Q}}|\mu(P \cap Q)-\mu(P) \mu(Q)|<\varepsilon
$$
Let 
\begin{equation}\label{definition of U}
    \mathcal{U}=\left\{U_1, \ldots, U_n\right\}\text{ be the partition of } \Sigma \text{ with }
U_j=\left\{x \in \Sigma: x_0=j\right\}.
\end{equation}

The partition $\mathcal{U}$ is called \textit{weak-Bernoulli} (for $\sigma$ and $\mu$) if for every $\varepsilon>0$ there is an $N(\varepsilon)$ so that
$$
\mathcal{P}=\mathcal{U} \vee \sigma^{-1} \mathcal{U} \vee \cdots \vee \sigma^{-s} \mathcal{U} \quad \text { and } \quad \mathcal{Q}=\sigma^{-t} \mathcal{U} \vee \cdots \vee \sigma^{-t-r} \mathcal{U}
$$
are $\varepsilon$-independent for all $s \geq 0, r \geq 0, t \geq s+N(\varepsilon)$. We say that an invariant measure $\mu$ is {\em weak
Bernoulli} if the standard partition $\mathcal{U}$ is weak Bernoulli. A well-known theorem of Friedman and Ornstein \cite{Friedman70} states that if $\mathcal{U}$ is weak-Bernoulli, then $(\sigma, \mu)$ is conjugate to a Bernoulli shift.

\subsection{Metrics on projective spaces}
\label{subsec:projective}

Let $d_{\P}$ be the metric on $\P$ given by 
$$
d_{\P}(\bar{u}, \bar{v})=\frac{\|u \wedge v\|}{\|u\|\|v\|}, \quad \text{for every $\bar{u}=\mathbb{R} u, \bar{v}=\mathbb{R} v \in \P$}.
$$
We recall that
for every $\bar{u}, \bar{v} \in \P$, we have 
\[\|\bar{u}-\bar{v}\|=2 \sin \left(\frac{\theta}{2}\right) \quad \text { and } \quad\|\bar{u} \wedge \bar{v}\|=\sin \theta\]
where $\theta \in [0, \pi)$  is the angle between the unit vectors $\bar{u}$ and $\bar{v}$.
Let $v^\perp \subset \P$ denote the projection onto projective space of the subspace orthogonal to a given vector $v$. Then the distance between $\bar{u}$ and $\bar{v}$ in $\P$ can be equivalently written as
\[
\delta_{\P}(\bar{u}, \bar{v}) = d_{\P}(\bar{u}, v^\perp) := \inf_{\bar{w} \in v^\perp} d_{\P}(\bar{u}, \bar{w}) .
\]

For any matrix $A \in \mathrm{GL}_d(\mathbb{R})$, let $\overline{v_{+}}(A)$ denote the element of $\mathbb{P}$ corresponding to the direction of the largest singular value of $A$, which in particular satisfies
$ 
\left\| A^* \frac{u}{\|u\|} \right\| = \left\| A^* \right\|,
$ 
where $A^*$ stands for the adjoint (i.e., transpose) of a matrix $A$.
\color{black}

We also define the \emph{spectral gap} of $A$ by
\[
\gamma_{1,2}(A) := \frac{\sigma_2(A)}{\sigma_1(A)},
\]
and observe that $\gamma_{1,2}(A^*) = \gamma_{1,2}(A)$. The following lemma will be instrumental in estimating the loss of multiplicativity in the inequality $\|Au\| \le \|A\| \|u\|$, for $u\in \mathbb R^d$. 

\begin{lem}[{\cite[Lemma 14.2]{benoist2016random}}]\label{benoist} For every $A \in \mathrm{GL}_d(\mathbb{R}), \bar{u}=\mathbb{R} u, \bar{v}=\mathbb{R} v \in \P$ one has
\begin{itemize}
  \item[(i)] $\delta_{\P}\left(\overline{v_{+}}\left(A^*\right), \bar{v}\right) \leq \frac{\|A v\|}{\|A\|\|v\|} \leq \delta_{\P}\left(\overline{v_{+}}\left(A^*\right), \bar{v}\right)+\gamma_{1,2}(A)$
  \item[(ii)] $\delta_{\P}\left(\bar{u}, \overline{v_{+}}(A)\right) \leq \frac{\left\|A^* u\right\|}{\|A\|\|u\|} \leq \delta_{\P}\left(\bar{u}, \overline{v_{+}}(A)\right)+\gamma_{1,2}(A)$.
  \item[(iii)] $d_{\P}\left(\overline{A^* u}, \overline{v_{+}}\left(A^*\right)\right) \delta_{\P}\left(\bar{u}, \overline{v_{+}}(A)\right) \leq \gamma_{1,2}(A)$.
\end{itemize}

\end{lem}

\section{Linear cocycles}
\label{subsecmatrixc}
\subsection{Continuous linear cocycles.}
We consider continuous \textit{linear cocycles (or linear cocycles)} $\mathcal{A}: \Sigma \rightarrow GL(d, \R)$ over $(\Sigma, \sigma)$.
Consider the skew-product $F:\Sigma \times \mathbb R^d \to \Sigma \times \mathbb R^d$ given by $F(x, v)=(\sigma(x), \mathcal{A}(x)v)$. For each $n\ge 1$, the $n$-th iterate of $F$ is  $F^n(x, v)=(\sigma^{n}(x), \mathcal{A}^{n}(x)v)$ 
where
$$
\mathcal{A}^{n}(x):=\mathcal{A}\left(\sigma^{n-1} (x)\right) \ldots \mathcal{A}(x).
$$
In case 
$\sigma$ is invertible then so is $F$ and $F^{-n}(x)=(\sigma^{-n}(x), \mathcal{A}^{-n}(x)v)$ for each $n\geq1$, where
\[\mathcal{A}^{-n}(x):=\mathcal{A}(\sigma^{-n}(x))^{-1}\mathcal{A}(\sigma^{-n+1}(x))^{-1} \dots \, \mathcal{A}(\sigma^{-1}(x))^{-1}.\]

Another system related to a given cocycle $\A: \Sigma \to \glr$ over a topological mixing subshift of finite type $(\Sigma, \sigma)$ is the skew product $F_{\mathcal{A}}: \Sigma \times \P \rightarrow \Sigma \times \P$ defined by
\begin{equation}\label{skew-product}
  F_{\mathcal{A}}(x, v):=(\sigma (x), \overline{\mathcal{A}(x) v}) .
\end{equation}

It is clear that the action of $\mathcal{A}^n(x)$ on the projective space $\P$ is encoded in the second coordinate of the iterations of the skew product $F_{\mathcal{A}}$. Similarly, we denote the skew product on $\Sigma^{+} \times \P$ by $\hat{F_{\mathcal{A}}}$. We also denote by $\M\left(F_{\A}\right)$  the set of Borel probability measures on $\Sigma \times \P$.

The submultiplicativity of the norm $\|\cdot\|$ implies that  $\| \A \|$ is submultiplicative in the sense that for any $m, n\ge 1$, and $x \in \Sigma$,
$$
0 < \|\mathcal{A}^{n+m}(x)\| \leq \|\mathcal{A}^{n}(T^{m}(x))\| \|\mathcal{A}^{n}(x)\|.
$$
We recall that
\begin{equation}
\label{eq:geometricpotential}
    \Phi_{\A}=\left\{\log \| \A^{n}\| \right\}_{n\ge 1}. 
\end{equation}
\subsubsection{Special classes of linear cocycles}
There are two special and natural classes of linear cocycles. The first is the class of \textit{one-step cocycles}, defined as follows.   Given a $k$-tuple of matrices $\textbf{A}=(A_{1},\ldots,A_{k})\in \glr^{k}$ , we associate with it the locally constant map $\mathcal{A}:\Sigma \rightarrow \glr$ given by $\mathcal{A}(x)=A_{x_{0}},$ that means the linear cocycle $\mathcal{A}$ depends only on the zero-th symbol $x_0$ of $(x_{l})_{l\in \Z}$. It is clear that 
$$
\mathcal A^n(x) = A_{x_{n-1}} \cdot \dots \cdot A_{x_{1}} \cdot A_{x_{0}}
$$
for any $x=(x_{n})_{n\in \Z}\in \Sigma$. One-step cocycles are usually considered the simplest class of linear cocycles and model random products of a finite collection of matrices.

\medskip
A second relevant class of linear cocycles is the so-called class of \emph{dynamical cocycles}, arising from smooth dynamical systems.
In fact,
if 
$f: M \rightarrow M$ is a smooth map of a closed Riemannian manifold $M$, the derivative cocycle $D f$ is a cocycle generated by the map $\mathcal{A}(x)=$ $D_x f: T_x M \rightarrow T_{f x} M$. Additionally, if the map $f$ is invertible, then $\A$ is an invertible cocycle over $f$.
More generally, one can consider for each $D f$-invariant sub-bundle $E \subset T M$, the linear cocycle $\left.D f\right|_E$, whose regularity is inherited from the regularity of the fiber-bundle $M\ni x \mapsto E_x$. 
\subsection{Slowest Oseledets' subspaces}

Given a $GL(d,\mathbb R)$-valued linear cocycle
$\mathcal{A}: \Sigma \rightarrow \text{GL}(d,\R)$ over a topologically mixing subshift of finite type $(\Sigma,\sigma)$ and $[I] \in \mathcal{L}$, we define
\begin{equation}\label{def-A(I)}
  \|\mathcal{A}(I)\|:=\max _{x \in[I]}\left\|\mathcal{A}^{|I|}(x)\right\|.
\end{equation}
The \emph{adjoint inverse cocycle} $\mathcal{A}_*^{-1}(x)$ is defined as
$$
\mathcal{A}_*^{-n}(x):=\left[\mathcal{A}^n(x)^{-1}\right]^*=\left[\mathcal{A}^n(x)^*\right]^{-1}.
$$
We also define
\begin{equation}
    \label{eq:deftransposeadjA}
\mathcal{A}^{[n]}(x):=\left[\mathcal{A}^n(x)\right]^*=\mathcal{A}(x)^* \ldots \mathcal{A}\left(\sigma^{n-1} x\right)^*.
\end{equation}

\begin{rem}
Note that 
$$
\mathcal{A}^{[n]}(x)=\mathcal{A}_*^{-n}(x)^{-1}
\quad \text{for every $n\ge 1$ and $x\in \Sigma$,} 
$$
a fact that will be used frequently throughout the paper.     
\end{rem}

\begin{defn}
\label{def:slowestOsel}
 For $x \in \Sigma$, we say that the \emph{slowest Oseledets' subspace} of the cocycle $\mathcal{A}_*^{-1}$ at $x$ is well defined if 
 \begin{enumerate}
     \item  $\lim _{n \rightarrow \infty} \frac{1}{n} \log \sigma_d\left(\mathcal{A}_*^{-n}(x)\right)$ exists; and
\item there exists a subspace $W_x \subseteq \mathbb{R}^d$ such that:
\begin{itemize}
    \item[(i) ]
$\lim _{n \rightarrow \infty} \frac{1}{n} \log \left\|\mathcal{A}_*^{-n}(x) u\right\|=\lim _{n \rightarrow \infty} \frac{1}{n} \log \sigma_d\left(\mathcal{A}_*^{-n}(x)\right)$
for all unitary $u \in W_x$; 
\item[(ii)] $\liminf_{n \rightarrow \infty}$ $\frac{1}{n} \log \left\|\mathcal{A}_*^{-n}(x) u \right\|>\lim _{n \rightarrow \infty} \frac{1}{n} \log \sigma_d\left(\mathcal{A}_*^{-n}(x)\right)$ for all unit vectors $u \notin W_x$, 
\end{itemize}
 \end{enumerate}   
 in which case $\overline{\xi_*}(x)$ denotes the projectivization of the subspace $W_x$ in $\P$.
\end{defn}
 
By Oseledets' theorem \cite{Oseledets}, one knows that 
given a $\Sigma$-invariant and ergodic probability measure $\mu$, the projective subsspace $\overline{\xi_*}(x)$ is well defined for $\mu$-a.e. $x\in \Sigma$, that 
$$
\lim _{n \rightarrow \infty} \frac{1}{n} \log \sigma_d\left(\mathcal{A}_*^{-n}(x)\right)=\lambda_d\left(\mu, \mathcal{A}_*^{-1} \right)=\lambda_1(\mu, \mathcal{A}),
$$
and that $\mathcal{A}_*^{-1}(x) \overline{\xi_*}(x)=\overline{\xi_*}(\sigma x)$ 
for $\mu$-a.e. $x \in \Sigma$.
Similarly, if \( \overline{\xi_*}(x) \) is well defined and  \( y \in \sigma^{-1}(x) \) then  \( \overline{\xi_*}(y) \) is well defined, 
and satisfies
$ 
\mathcal{A}_*^{-1}(y)\, \overline{\xi_*}(y) = \overline{\xi_*}(x).
$ 

\medskip
In the case where $\overline{\xi_*}(x)$ is 0-dimensional (i.e., the subspace $ W_x\subset \mathbb R^d$ from Definition~\ref{def:slowestOsel} is one-dimensional), we choose and fix a vector 
$ \xi_*(x) \in W_x$
such that $ \|\xi_*(x)\| = 1 $.
Define the set
\begin{equation}
    \label{eq:defY}
Y := \left\{ x \in \Sigma : \overline{\xi_*}(x) \text{ is well defined and 0-dimensional} \right\}.
\end{equation}
Note that $\sigma^{-1}(Y)=Y$. The latter fact will be used later.

\subsection{Fiber bunching and holonomies}\label{subsec:fiberbunching}
 Given $\theta>0$, a linear cocycle $\mathcal{A}:\Sigma \rightarrow GL(d, \R)$ is \emph{$\theta$-H\"older continuous}  if there exists $C_0>0$ such that
\begin{equation}\label{hol}
 \|\mathcal{A}(x)-\mathcal{A}(y)\|\leq C_0d(x,y)^{\theta} \hspace{0,2cm} \quad \forall x,y \in \Sigma.
\end{equation}
We denote by $C^\theta(\Sigma, G L(d, \mathbb{R}))$ the vector space of $\theta$-H\"older continuous linear cocycles over the topologically mixing subshift of finite type $(\Sigma, \sigma)$.

\smallskip
We say that the cocycle $\mathcal{A} \in C^{\theta}(\Sigma, GL(d, \R))$ over $(\Sigma, \sigma)$ is \textit{fiber-bunched} if 
\begin{equation}\label{fiber}
\|\mathcal{A}(x)\|\|\mathcal{A}(x)^{-1}\|< 2^{\theta}
\quad\text{for every $x\in \Sigma$}
\end{equation} 
(note that the constant $2$ appearing above is a hyperbolicity constant for the shift map $\sigma$).

\begin{rem}
\label{rmk:pert-conformal}
A conformal cocycle $\A$ (i.e. $\A(x)$ is a scalar multiple of an orthogonal matrix, for all $x\in \Sigma$) verifies $\|\mathcal{A}(x)\| \cdot\left\|\mathcal{A}(x)^{-1}\right\|=1$ for all $x \in \Sigma$. Therefore, small $C^0$-perturbations of H\"older continuous conformal cocycles are fiber-bunched.    
\end{rem}

Let us denote the \emph{set of fiber-bunched cocycles}
by 
$$
C_b^{\theta}\left(\Sigma, \mathrm{GL}_d(\mathbb{R})\right):=\left\{\mathcal{A} \in C^{\theta}\left(\Sigma, \mathrm{GL}_d(\mathbb{R})\right): \mathcal{A} \text { is fiber-bunched}\right\}.
$$
It is clear that the latter is a $C^0$-open subset of $C^{\theta}\left(\Sigma, \operatorname{GL}_d(\mathbb{R})\right)$.



\begin{defn}\label{holonomy}
Given $\mathcal A\in C_b^{\theta}(\Sigma, GL(d, \R))$,
$x\in \Sigma$ and $y\in W_{\loc}^{s}(x)$
the \textit{local stable holonomy} 
$H_{y \leftarrow x}^{s} \in GL(d, \R)$ is defined
by the limit (in case the limit exists)
$$
H_{y \leftarrow x}^{s}:=\lim _{n \rightarrow+\infty} \mathcal{A}^n(y)^{-1} \mathcal{A}^n(x).
$$
Similarly, the \emph{local unstable holonomy} $H_{y \leftarrow x}^{u}$ is likewise defined as (in case the limit exists)
$$
H_{y \leftarrow x}^{u}:=\lim _{n \rightarrow-\infty} \mathcal{A}^n(y)^{-1} \mathcal{A}^n(x)
$$
for any $x\in \Sigma$ and $y\in W_{\loc}^{u}(x)$.
\end{defn}

Some comments are in order. First, even though the stable and unstable holonomies depend on the linear cocycle $\mathcal A$, we shall omit its dependence on $\mathcal A$ 
for notational simplicity. Second, it follows from \cite{bonatti2004lyapunov} that stable and unstable holonomies exist for fiber-bunched linear cocycles, that stable holonomies satisfy
\begin{itemize}
\item[a)]$H_{x \leftarrow x}^{s}=Id$ and $H_{z \leftarrow y}^{s} \circ H_{y \leftarrow x}^{s}=H_{z \leftarrow x}^{s}$ for any $z,y \in W_{\loc}^{s}(x)$.
\item[b)] $\mathcal{A}(y)\circ H_{y \leftarrow x}^{s}=H_{\sigma(y) \leftarrow \sigma(x)}^{s}\circ \mathcal{A}(x).$
\item[c)] $(x, y, v)\mapsto H^s_{y\leftarrow x}(v)$ is continuous.
\end{itemize}
and that similar properties hold for unstable holonomies (with $s$ and $\sigma$ replaced by $u$ and $\sigma^{-1}$, respectively).
Third, one can use item $b)$ above to one can extend the definition to the global stable holonomy $H_{y\leftarrow x}^{s}$ for points $y\in W^{s}(x)$ 
as
\begin{equation}\label{extension of holonomy}
H_{y\leftarrow x}^{s}=\mathcal{A}^{n}(y)^{-1} \circ H_{\sigma^{n}(y)\leftarrow \sigma^{n}(x)}^{s}\circ \mathcal{A}^{n}(x),
\end{equation} 
where 
$n\in \N$ is large enough such that $T^{n}(y)\in W_{\loc}^{s}(T^{n}(x))$ (the global unstable holonomy can be defined similarly).
Finally, the 
holonomies vary $\theta$-H\"older continuously (see \cite{KS}), which means that there exists a constant $C > 0$ such that 
\begin{itemize}
\item[d)] $\|H_{y \leftarrow x}^{s} - I\| \leq C d(x,y)^{\theta}$ for every $y \in W_{\loc}^{s}(x)$ (and an analogous statement for unstable holonomies).
\end{itemize}

\begin{rem}
In the simpler context of one-step cocycles, stable and unstable holonomies always exist (see \cite[Proposition 1.2]{bonatti2004lyapunov} and \cite[Remark 1]{Moh22}).
\end{rem}

The existence and continuity of the canonical holonomies guarantee that $\mathcal{A}$ has the following bounded distortion property: there exists $C>0$ such that for any $n\ge 1$ and $x, x \in \Sigma$ with $y \in[x]_n$, we have
\begin{equation}\label{BDD-dis}
C^{-1} \leq \frac{\left\|\mathcal{A}^n(x)\right\|}{\left\|\mathcal{A}^n(y)\right\|} \leq C.
\end{equation}

Lastly, we show that every fiber-bunched linear cocycle $\mathcal{A}: \Sigma \rightarrow \mathrm{GL}_d(\mathbb{R})$ over the subshift $(\Sigma,\sigma)$ is cohomologous to a linear cocycle over the one sided subshift $(\Sigma^+,\sigma)$.
Recall that $\Sigma \subset \{1,2,\dots,k\}^{\Z}$.
For each $1\le i \le k$, fix $\eta^i \in \Sigma$ with $\left(\eta^i\right)_0=i$. Given $x=\left(x_i\right)_{i \in \mathbb{Z}} \in \Sigma$, set
$ 
x_\eta:=\left[\eta^{x_0}, x\right].
$ 
Then consider the cocycle $\hat{\mathcal{A}}: \Sigma^+ \rightarrow \mathrm{GL}_d(\mathbb{R})$ given by
$$
\hat{\mathcal{A}}(x):=\mathcal{A}\left(x_\eta\right) .
$$
From its definition, $\hat{\mathcal{A}}$ is constant along the local stable set, that is,
\begin{equation}\label{stable holonomy}
\hat{\mathcal{A}}(x)=\mathcal{A}(y) \text { for every } y \in W_{\mathrm{loc}}^s(x).
\end{equation}
Set $\mathcal{C}(x):=H^s_{x_\eta \leftarrow x}$, where $H^s$ denotes local stable holonomies of 
$\mathcal{A}$. Notice that the cocycle $\mathcal C$ serves as a conjugacy between $\mathcal{A}$ and $\hat{\mathcal{A}}$, as
\begin{equation}
\label{eq:cohomology}
\hat{\mathcal{A}}(x) \circ \mathcal{C}(x)=\mathcal{C}(\sigma x) \circ \mathcal{A}(x)
\text { for every } x \in \Sigma.
\end{equation}
(we refer the reader e.g. to \cite[Corollary 1.15]{Bonatti2003} for further details).
It is clear from the construction that the cocycle $\hat \A$ is such that its stable holonomies coincide with the identity at all points, hence it can be identified with a cocycle over the one-sided subshift ($\Sigma^{+}, \sigma$).

\medskip
{\bf Throughout the remainder of the paper} we will assume that all fiber-bunched cocycles are constant along stable manifolds as described above and, by some abuse of notation, use $\mathcal{A}$ to denote the cocycles over both the two-sided and one-sided subshifts ($\Sigma, \sigma$) and ($\Sigma^{+}, \sigma$), respectively. There is no loss of generality 
as if the cohomological equation ~\eqref{eq:cohomology} holds then it is not hard to check that $\A$ and $\hat\A$ share the same Lyapunov exponents and pressure functions.

\subsection{Typical cocycles}
\label{subsec:typical} 
Let $\sigma:\Sigma \to \Sigma$ be a topologically mixing subshift of finite type.
 Assume that $p\in \Sigma$ is a periodic point of $\sigma$. We say $z\in \Sigma\setminus\{p\}$ is a \textit{homoclinic point} associated to $p$ if 
 $z\in W^{s}(p) \cap W^{u}(p)$. We denote the set of all homoclinic points of $p$ by
$\mathcal{H}(p)$. For each 
$z\in W_{\text{loc}}^{s}(p) \cap W_{\text{loc}}^{u}(p)$
consider the matrix 
\begin{equation}\label{eq:matrixtildeH}
 {\widetilde H}_{p}^z :=H_{p \leftarrow z}^s \circ H_{z \leftarrow p}^u    
\end{equation}
associated to the homoclinic loop. Such matrices have been used in a crucial way to prove simplicity of the Lyapunov spectrum for typical fiber-bunched linear cocycles (cf. \cite{bonatti2004lyapunov}).
In general, given $z\in \mathcal H(p)$, up to replacing $z$ by some backward iterate, we may Assume that $z\in W_{\text{loc}}^{u}(p)$ and $T^{n}(z)\in W_{\text{loc}}^{s}(p)$ for
some $n \geq 1$. 
Then, using 
\eqref{extension of holonomy}, 
one defines
\begin{equation}
\label{eq:analogueeq}
 {\widetilde H}_{p}^z =\mathcal{A}^{-n}(p)\circ H_{p \leftarrow \sigma^{n}(z)}^{s} \circ \mathcal{A}^{n}(z) \circ H_{z \leftarrow p}^{u}.    
\end{equation}

\begin{defn}\label{typical1}
Assume that $\mathcal{A} \in C_{b}^{\theta}(\Sigma, GL(d, \R))$. We say that $\mathcal{A}$ is \textit{1-typical} with respect to a pair $(p,z)$, where $p$ is a  periodic point for $T$ and $z\in \mathcal H(p)$ if:
\begin{itemize}
\item[(i)]  (\emph{pinching}) \color{black} 
the eigenvalues of  $\mathcal{A}^{per(p)}(p)$ have multiplicity $1$ and distinct absolute values,
\item[(ii)] (\emph{twisting})   the eigenvectors $\left\{v_{1}, \ldots, v_{d}\right\}$ of $\mathcal{A}^{per(p)}(p)$ are such that, for any $I, J \subset \{1, \ldots, d\}$ with $|I|+$ $|J| \leq d$, the set of vectors
$$
\left\{ {\widetilde H}_{p}^z  \left(v_{i}\right): i \in I\right\} \cup\left\{v_{j}: j \in J\right\}
$$
is linearly independent.
\end{itemize}
\end{defn}

\begin{defn}
    We say $\mathcal{A} \in C_{b}^{\theta}(\Sigma, GL(d, \R))$ is \textit{typical} if $\mathcal{A}^{\wedge t}$ is 1-typical with respect to the same pair $(p, z)$ for all $1 \leq t \leq d-1$.
\end{defn}

\begin{rem}\label{fixed point} For simplicity, we will always assume that $p$ in Definition~\ref{typical1} is a fixed point by considering the map $\sigma^{\text{per}(p)}$ and the cocycle $\mathcal{A}^{\text{per}(p)}$ if necessary (this is possible because powers of 1-typical cocycles are 1-typical). Moreover, for any homoclinic point $z \in \mathcal{H}(p)$, $\sigma^n (z)$ is a homoclinic point of $p$ for any  $n \in \mathbb{Z}$. 
\end{rem}

\begin{rem}
Observe that  $\mathcal{A}_*^{-1}$ is a 1-typical cocycle over $\sigma$ whenever $\A$ is a 1-typical cocycle (see \cite[Lemma 7.2]{bonatti2004lyapunov}).    
\end{rem}

\subsection{Non-additive thermodynamic formalism}\label{Non-additive TF}

In the context of linear cocycles, it appears naturally some sequences of sub-additive real-valued cocycles. In what follows we recall some notions from the thermodynamic formalism of non-additive sequences of potentials.

Let $(X, d)$ be a compact metric space and $T:X\rightarrow X$ be a continuous map. 
Recall that a family $\Phi=\{\phi_{n}\}_{n=1}^{\infty}$ of continuous potentials over the topological dynamical system $(X, T)$ is called \emph{sub-additive} 
(resp. \emph{super-additive}) if 
$\phi_{m+n}(x)
\le \phi_{m}(x)+\phi_n(T^m(x))$ 
(resp. $\phi_{m+n}(x)
\ge \phi_{m}(x)+\phi_n(T^m(x))$) 
for every $x\in\Sigma$ and $m,n\ge 1$.

\medskip
Let $\Phi=\{\phi_{n}\}_{n=1}^{\infty}$ be either a sub-additive or super-additive family of continuous potentials over a topological dynamical system $(X,T)$ and,
for each $n\in \N$, consider the metric $d_{n}$ on $X$ given by
\begin{equation}\label{new_metric}
 d_{n}(x, y)=\max\Big\{\, d(T^{k}(x), T^{k}(y)) : 0\le k \le n-1\,\Big\}.
\end{equation}
Given $\varepsilon>0$ and $n\ge 1$ we say that $E \subset X$ 
is an $(n,\varepsilon)$-\textit{separated  subset} if $d_{n}(x,y)> \varepsilon$ 
for any two points $x\neq y \in E$, and define
\[ P_{n}(T, \Phi, \varepsilon)=\sup \bigg\{\sum_{x\in E} e^{\phi_{n}(x)} : E \hspace{0,1cm}\textrm{is} \hspace{0,1cm}(n, \varepsilon) \textrm{-separated subset of }X \bigg\}.\]
The $\textit{topological pressure}$ of $\Phi$ is defined by
\begin{equation}\label{epsilon}
P(T,\Phi)=\lim_{\varepsilon \rightarrow 0}
\limsup_{n \to +\infty} \frac{1}{n} \log P_{n}(T, \Phi, \varepsilon),
\end{equation} 
where the limit in $\vep$ exists by monotonicity. 
%
\begin{rem}\label{(n, 1)separated sets}
If $(\Sigma, \sigma)$ is a subshift of finite type 
endowed with the metric $d$ defined in ~\eqref{metric} then $T$ is an expansive map with expansivity constant $\vep_0$ (meaning that for any $x\neq y\in \Sigma$ there exists $n\in \mathbb Z$ so $d(\sigma^n(x), \sigma^n(y))\ge \vep_0$) then the pressure $P(\sigma,\Phi)$ coincides with 
\[
 P^*(\Phi)  =\limsup _{n \rightarrow \infty} \frac{1}{n} \log \sup \left\{\sum_{x \in E}  e^{\phi_{n}(x)}:E \hspace{0,1cm}\textrm{is} \hspace{0,1cm}(n, \vep_0) \textrm{-separated subset of }\Sigma \right\}
\]
(see e.g. \cite{Walters}).
In consequence, if
$\A:\Sigma \to \glr$ is a fiber-bunched linear cocycle over a topologically mixing subshift of finite type $(\Sigma, \sigma)$, by bounded distortion (cf.~\eqref{BDD-dis}), one has
\begin{equation}\label{the definition of the topological pressure for fb-matrix}
    P(\sigma, t\Phi_{\A})=\lim_{n\to \infty}\frac{1}{n}\log \sum_{I \in \Sigma_n}\|\A(I)\|^t
\end{equation}
for every $t \in \R$ (notice $\Phi_\A$ is the family defined in 
~\eqref{eq:geometricpotential} and the limit exists by monotonicity). 
\end{rem}

\medskip
Cao, Feng and Huang \cite{CFH} proved a variational principle formula for the topological pressure of sub-additive potentials, while the counterpart for super-additive potentials was established by Cao, Pesin and Zhao \cite{CPZ}. 
More recently, in \cite{Mohammadpour-Varandas} we proved a variational principle for the generalized singular value function, which is a generalization of the family of potentials $\Phi_\A$ and is neither sub-additive nor supper-additive (we refer the reader to \cite[Theorem B]{Mohammadpour-Varandas} for more details).
Hence, for any $t \in \R$,
\begin{equation}\label{varitional}
P(T,t\Phi)=\sup \bigg\{h_{\mu}(T)+t\lambda(\mu, \Phi): \mu \in \mathcal{M}_{\text{inv}}(T) \bigg\},
\end{equation}
where $h_{\mu}(T)$ is the measure-theoretic entropy of $\mu$ and 
\begin{equation}\label{defchimu}
    \lambda(\mu, \Phi):=\lim_{n\to \infty}\frac1n \int \phi_{n}(x) \; d\mu(x).
\end{equation}

\medskip
Any invariant measure $\mu \in \mathcal{M}_{\text{inv}}(T)$ achieving the
supremum in \eqref{varitional} is called an \textit{equilibrium state (measure)} of $t\Phi$. In particular, we say that ${\hat \mu_t}$ is an {\em equilibrium state} for $t\Phi$ if 
\begin{equation}\label{equilibrium state}
    P(T,t\Phi )= h_{{\hat \mu_t}}(T)+t\lambda({\hat \mu_t}, \Phi).
\end{equation}

\begin{rem} For $t\geq 0$, the functional $\mu\mapsto \lambda(\mu,t\Phi)$ is upper semi-continuous (in the weak$^{*}$ topology). Hence, if the entropy $\operatorname{map} \mu \mapsto h_\mu(T)$ is upper semi-continuous (e.g. in case $T$ is a subshift of finite type), then there exists at least one equilibrium state for each sub-additive family of continuous potentials $t\Phi$. However, if $t<0$ the Lyapunov exponent functional 
$\mu\mapsto \lambda(\mu,t\Phi)$
is lower semi-continuous, and equilibrium states may fail to exist even if the entropy map is upper semi-continuous.
\end{rem}


\begin{defn}
\label{def:Gibbsmeasure}
 We say that a probability measure ${\hat \mu_t} \in \M(\sigma)$ is a {\em Gibbs measure} for $\Phi=\{\phi_n\}_{n=1}^\infty$ 
 if there exist $C_1, C_2>0$ and $P \in \mathbb{R}$ such that for any $n \in \N$ and $I \in \Sigma_n$
 \begin{equation}\label{orginal def of Gibbs}
   C_1 \leq \frac{{\hat \mu_t}\left([I]\right)}{\exp(-P n + \phi_n(x)) } \leq C_2  
\end{equation}
for any $x \in [I].$
\end{defn}

Assume that $t \in \R.$ Let $\mathcal{A}$ be a fiber-bunched cocycle and let $f$ be a Hölder continuous function over a topologically mixing subshift of finite type $(\Sigma^+, \sigma)$. We define the potential $\Phi_{\mathcal{A}, f, t} := \left(\phi_{t, n}^f\right)_{n \geq 1}$, where $$
\phi_{t, n}^f:=S_n f(x)+t \log \left\|\mathcal{A}^n(x)\right\|.$$

\begin{lem}\label{Gibbs measure is Eq and the topological pressure}
 Assume that a $\sigma$-invariant probability measure ${\mu_t}$ is a Gibbs measure for $\Phi_{\A, f, t}$. Then: (i) the constant $P$ in \eqref{orginal def of Gibbs} is equal to $P(\sigma, \Phi_{\mathcal{A}, f, t})$, and (ii) ${\mu_t}$ is an equilibrium state for $\Phi_{\mathcal{A}, f, t}.$
\end{lem}
\begin{proof}
Assume that ${\mu_t}$ is a $\sigma$-invariant probability measure satisfying 
\eqref{orginal def of Gibbs}, for some $P\in\mathbb R$.
By \eqref{the definition of the topological pressure for fb-matrix} and the bounded distortion for $f$ and \eqref{BDD-dis},
$$
    \begin{aligned}
\frac{1}{n} \log \sum_{I \in \Sigma_n}\sup_{x \in [I]} e^{S_{n}f(x)}\|\A^n(x)\|^t 
& \geq  \frac{1}{n} \log \sum_{I \in \Sigma_n} \frac1{C_2 C^t} {\hat \mu_t}([I]) e^{P n} \\
& =P - \frac{1}{n} \log [C_2 C^t]
\end{aligned}
$$
and, similarly, 
$$\frac{1}{n} \log \sum_{I \in \Sigma_n}  \sup_{x \in [I]} e^{S_{n}f(x)}\|\A^n(x)\|^t  \leq P+\frac{1}{n} \log [C_1C^t].$$
 Therefore, $P(\sigma,\Phi_{\mathcal{A}, f, t})=P$.
Finally, by the Gibbs property,
$$
\begin{aligned}
h_{{\mu_t}}(\sigma)+t \lambda({ \mu_t}, \A)+\int f d\mu_t& \geq \limsup _{n \rightarrow \infty} \frac{1}{n} \sum_{I \in \Sigma_n} {\mu_t}([I])(-\log {\mu_t}([I])+\sup_{x \in [I]} (e^{S_{n}f(x)}\|\A^n(x)\|^t)) \\
& \geq \limsup_{n \rightarrow \infty} \frac{1}{n} \sum_{I \in \Sigma_n} { \mu_t}([I])\left(n P(\sigma,\Phi_{\mathcal{A}, f, t})-\log C_2\right) \\
& =P(\sigma,\Phi_{\mathcal{A}, f, t}),
\end{aligned}
$$
thus proving that ${\mu_t}$ is an equilibrium state for $\Phi_{\mathcal{A}, f, t}$. 
\end{proof}

\section{Transfer operator}\label{sec: Transfer}
In this section we will discuss a special class of transfer operators in the context of $\theta$-H\"older continuous linear cocycles and study the properties of special eigenmeasures.

\color{black}

\subsection{Equilibrium states for the subshift}
\label{subsec-psi}
In what follows let us recall some notions from the classical thermodynamic formalism, referring the reader to \cite{Bow} for details.

\medskip
Throughout let $\psi: \Sigma \rightarrow \mathbb{R}$ be a fixed $\alpha$-H\"older continuous function. It is well known that $\psi$ is cohomologous to a H\"older continuous potential $\hat\psi$ that depends only on positive coordinates, hence $\hat{\psi}$ can be identified with a H\"older continuous potential on $\Sigma^{+}$. Moreover, 
if $h_{\hat \psi}$ is the non-negative $\alpha$-H\"older continuous leading eigenfunction for the transfer operator 
\begin{equation}
    \label{eq:standardRPF}
\L_{\hat{\psi}} f(x):=\sum_{y \in \sigma^{-1} x} e^{\hat{\psi}(y)} f(y),
\end{equation}
acting on a suitable space of H\"older continuous functions on $\Sigma^+$, then 
\begin{equation}
\label{eq:g-functiondef}
g(x):=\frac{e^{\hat{\psi}(x)}}{e^{P(\sigma,\hat{\psi})}} \cdot \frac{h_{\hat{\psi}(x)}}{h_{\hat{\psi}}(\sigma x)},    
\end{equation}
then $g$ is a H\"older continuous $g$-function cohomologous to $\hat{\psi}$, $g(x)>0$ for every $x\in \Sigma$,  and 
$ 
\sum_{y \in \sigma^{-1} x} g(y)=1
\quad\text{for all $x \in \Sigma^+$}
$ 
($g$ is called a $g$-function and often one writes  $g \in \mathcal{G}\left(\Sigma\right)$).
Let $\hat{\mu}$ be the unique equilibrium state for $\hat{\psi}$ on $(\Sigma^+,\sigma)$. 
There exists a unique probability measure $\mu$ on $\Sigma$ so that $\Pi_*\mu=\hat \mu$; moreover, it is the unique equilibrium state for $\psi$. 
The probability measure $\hat{\mu}$ is known to be a Gibbs measure: 
There exists $C>0$ such that
$$ 
C^{-1} g^{(n)}(x) \leq \hat{\mu}([I]) \leq C g^{(n)}(x)
$$
for every $I\in \Sigma_n$ and $n\ge 1$, where $$g^{(n)}(x):=g\left(\sigma^{n-1} x\right) \ldots g(x).$$
We refer the reader to \cite{Bow} for more details. 

\medskip
The existence of hyperbolic invariant manifolds allows us to further describe the probability measure $\mu$. Indeed, by considering $\Sigma^{+}$ as the parameter space for the local stable sets of $\Sigma$, by Rohlin's disintegration theorem there exists a disintegration $\left\{{\mu}_x^s\right\}_{\hat{x} \in \Sigma^{+}}$ of $\mu$ where each $\mu_{\hat{x}}^s$ is a probability measure supported on $W_{\text {loc }}^s(\hat{x})$ for $\hat{\mu}$-a.e. $\hat{x}$, and 
$$
\mu=\int_{\Sigma^{+}} \mu_{\hat{x}}^s \; d \mu(\hat{x}).
$$

Moreover, such equilibrium states $\mu$
have \emph{local product structure}:
There exists $K_0>0$ so that 
\begin{equation}
    \label{eq:LPSm}
K_0^{-1} \leq \frac{\mu([x_{-n} \dots x_{m-1}])}{\mu([x_{-n} \dots x_{-1}]) \cdot \mu([x_0 \dots x_{m-1}])} \leq K_0 
\end{equation}
for every $m,n\ge 1$ and $x=(x_n)_{n\in \mathbb Z} \in \Sigma \cap \text{supp}(\mu)$. 
\color{black}
Furthermore, as equilibrium states of H\"older potentials have continuous local product structure, the 
unstable holonomy between local stable sets defined by
$$
\begin{aligned}
h_{\hat{x}, \hat{y}}:\left(W_{\mathrm{loc}}^s(\hat{x}), \hat{\mu}_{\hat{x}}^s\right) & \rightarrow\left(W_{\mathrm{loc}}^s(\hat{y}), \hat{\mu}_{\hat{y}}^s\right) \\
x & \mapsto[x, \hat{y}]
\end{aligned}
$$
is absolutely continuous for every $x, y \in \Sigma$ with $x_0=y_0$. Furthermore, its Jacobian $J_{x, y}$ depends H\"older continuously in $\hat{x}, \hat{y} \in \Sigma^{+}$, and 
$ 
\hat{\mu}_{\hat{y}}^s=J_{\hat{x}, \hat{y}}\left(h_{\hat{x}, \hat{y}}\right)_* \hat{\mu}_{\hat{x}}^s .
$ 
(cf. \cite[Lemma~2.6]{bonatti2004lyapunov} and references therein).

\subsection{Holonomy invariant measures}\label{subse: holonomy}
Let $\mu$ be the equilibrium state with respect to the potential $\psi$, as described in Subsection~\ref{subsec-psi}. We will recall some known results, due to Bonatti and Viana \cite{bonatti2004lyapunov}, on holonomy invariant measures.  

\medskip Let $\mathcal{A}: \Sigma \rightarrow \glr$ be a fiber-bunched cocycle, and $m$ be a $F$-invariant measure on $\Sigma \times \P$ that projects to $\mu$ under the canonical projection $\pi: \Sigma \times \P \to \Sigma$. 
Also, we denote $\pi^+: \Sigma ^+\times \P \to \Sigma^+$ the canonical projection.

\begin{defn}
    A probability measure $m$ on $\Sigma \times \P$ is \emph{$H^u$-invariant} if $\pi_* m=\mu$ 
    and there exists a disintegration $\left\{m_{x}\right\}_{x\in \Sigma}$ along the fibers such that
$$
\left(H_{x, y}^u\right)_* m_{x}=m_{y}
$$
for every $x$ and $y$ in the same local unstable set. We say $m$ is 
\emph{$\left(\mathcal{A}, H^u\right)$-invariant} (also known as $u$-state) if it is $H^u$-invariant and, in addition, $\hat{m}$ is $\hat{F}_{\mathcal{A}}$-invariant. 
We say that a probability measure $\hat{m}$ on $\Sigma^{+} \times \P$ is \emph{$\left(\mathcal{A}, H^u\right)$-invariant} if there exists an $\left(\mathcal{A}, H^u\right)$-invariant probability measure $m$ on $\Sigma \times \P$ with $\pi^+_* m=\hat{m}$.
\end{defn}

\begin{rem}
The notion of $\left(\mathcal{A}, H^s\right)$-invariance can be defined analogously. Moreover, while the existence of $\left(\mathcal{A}, H^u\right)$-invariant measures is not obvious, it was shown in \cite{bonatti2004lyapunov} that the set of $\left(\mathcal{A}, H^u\right)$-invariant measures is non-empty.    
\end{rem}

The main result in \cite{bonatti2004lyapunov} ensures that if $\mathcal{A}$ is 1-typical, then the top and bottom Lyapunov exponents of $\A$ with respect to $\mu$ are simple. Let $\overline{\xi(x)} \in \P$  be the
projectivization of the top Oseledets subspace at $x$ with respect to $\A$ and $\mu$; when there
is no confusion, we will denote a unit vector in $\mathbb R^d$ in its direction also by $\xi(x)$.
It is known that there exists a unique ($\mathcal{A}, H^u$)-invariant measure over the probability measure $\mu$ with local product structure. More precisely: 

\begin{thm}\label{BV's results}
 Suppose $\mathcal{A}$ is a 1-typical fiber-bunched cocycle and let $m$ be an $F$-invariant probability measure on $\Sigma \times \P$ which is $H^u$-invariant  and projects to $\mu$, let $\hat m =\pi^+_*m$ and $\hat\mu=\Pi_*\mu$. The following properties hold:
 \begin{itemize}
     \item[1.] The probability measure $\hat{m}$ on $\Sigma^+\times\P$ admits a continuous disintegration $\left\{\hat{m}_x\right\}_{x \in \Sigma^{+}}$, i.e. so that $\Sigma^+ \ni x \mapsto \hat{m}_x$ is continuous (in the weak$^*$ topology);
\item[2.] For $\hat{\mu}$-a.e. $x \in \Sigma^{+}$, we have

$$
\sum_{\sigma y=x} \frac{1}{J_{\hat{\mu}} \sigma(y)} \mathcal{A}(y)_* \hat{m}_y=\hat{m}_x
$$
where $J_{\hat{\mu}} \sigma: \Sigma^{+} \rightarrow(0, \infty)$ is the Jacobian for $\hat{\mu}$;
\item[3.] For any $x \in \Sigma^{+}$ and proper projective subspace $V \subset \P$, we have $\hat{m}_x(V)=0$;
\item[4.] For ${\mu}$-a.e. $x \in \Sigma$
$$
m_x=\lim_{n\to\infty} \A^n(x_n)_* \hat m_{x_n}
\quad\text{where}\;\; 
x_n= \Pi(\sigma^{-n}(x)) \in \Sigma^+;
$$
\item[5.] for every $x \in \Sigma^{+}$, we have
\begin{equation}\label{conditional measures}
\hat{m}_{x}=\int m_{y} \, d \mu_{x}^s(y)=\int \delta_{\xi(y)} \; d \mu_x^s(y),
\end{equation}
where $(\mu_x^s)_{x\in \Sigma^+}$ is a disintegration of $\mu$ on the partition $\{W^s_{\text{loc}}(x)\colon x\in \Sigma^+\}$.
 \end{itemize}
\end{thm}

\begin{proof}
Item (1) corresponds to \cite[Proposition 4.3]{bonatti2004lyapunov}, item (2) follows from 
\cite[Remark 4.8]{bonatti2004lyapunov}, item (3)
was proved in \cite[Proposition~5.1]{bonatti2004lyapunov}
item (4) is corresponds to
\cite[Proposition~3.1 and Remark 3.5]{bonatti2004lyapunov}, and item (5) follows as a consequence of \cite[Lemma 4.7 and Theorem~3]{bonatti2004lyapunov}.
\end{proof}

\subsection{Transfer operators adapted to linear cocycles}
\label{subsec:53}

To ease notation, whenever no confusion arises, we henceforth denote elements of the one-sided symbolic space by $x$ rather than $\hat{x}$.
We endow $\Sigma \times \P$ and $\Sigma^+ \times \P$ with the metric
\begin{equation}\label{definition of the metric}
d_{\Sigma^+ \times \P}((x, \bar{u}),(y, \bar{v}))=\max \left\{d(x, y), d_{\P}(\bar{u}, \bar{v})\right\}
\end{equation}
where $d$ is the metric in $\Sigma$ (resp. $\Sigma^+$) and $d_{\P}$ stands for the metric in the projective space, introduced in Subsection~\ref{subsec:projective}.
For $0<\alpha \leq 1$, let $C^\alpha\left(\Sigma^+ \times \P\right)$ denote the vector space of real-valued $\alpha$-H\"older continuous maps $f: \Sigma^+ \times \P \rightarrow \mathbb{R}$, endowed with the norm 
$$
\|f\|_\alpha:=|f|_\alpha+\|f\|_{\infty}
$$
for each $f \in C^\alpha\left(\Sigma^+ \times \P\right)$, where
$$
|f|_\alpha:=\sup _{(x, \bar{u}) \neq(y, \bar{v})} \frac{|f(x, \bar{u})-f(y, \bar{v})|}{d_{\Sigma^+ \times \P}((x, \bar{u}),(y, \bar{v}))^\alpha}.
$$

\medskip 
\color{black}

 \medskip
Given a $\theta$-H\"older continuous 1-typical fiber-bunched cocycle $\mathcal{A}: \Sigma \rightarrow \mathrm{GL}_d(\mathbb{R})$ we will be interested in the following three families of continuous potentials:
\begin{itemize}
    \item $t\Phi_{\A}$, where $\Phi_{\A}$ is defined as in \eqref{eq:geometricpotential},
    \item $\Phi_{\A, t}=\left(\phi_{t, n}\right)_{n\ge 1}$, where $$
\phi_{t, n}:=S_n \psi(x)+t \log \left\|\mathcal{A}^n(x)\right\|
.$$
where $\psi: \Sigma \to \mathbb R$ is the H\"older continuous potential as in Subsection~\ref{subsec-psi},
    \item $\Phi_{\A, t}^g:=\left(\phi_{t, n}^g\right)_{n\ge 1}$ defined by
$$
\phi_{t, n}^g(x):=\log g^{(n)}(x)+t \log \left\|\mathcal{A}^n(x)\right\|,
$$
where $g\in \mathcal{G}\left(\Sigma\right)$ is defined by ~\eqref{eq:g-functiondef}.
\end{itemize}

\smallskip
It follows from the definition of $g$ in ~\eqref{eq:g-functiondef} and definition of Gibbs measures in ~\eqref{orginal def of Gibbs} that Gibbs measures for $\Phi_{\A, t}^g$ and $\Phi_{\A, t}$ coincide. Moreover, 
\begin{equation}\label{relation between pressures of potentials}
    P\left(\sigma,\Phi_{\A, t}^g\right)=P\left(\sigma,\Phi_{\A, t}\right)-P(\sigma,\psi).
\end{equation}

\color{black}
\begin{defn}
    Consider the transfer operator $\mathcal{L}: C^\alpha(\Sigma^+\times \P) \to C^\alpha(\Sigma^+\times \P)$
    given by 
$$
\mathcal{L} f(x, \bar{u}):=\sum_{y: \sigma y=x} g(y) f\left(y, \overline{\mathcal{A}(y)^* u}\right).
$$
\end{defn}

 \begin{rem}
It is simple to check that the spectral radius $\rho(\mathcal{L})$ is equal to 1 and the constant function 1 is an eigenfunction of $\mathcal{L}$.  \end{rem}

\smallskip

\color{black}
\begin{defn}
   Given $t\in \mathbb R$, consider the trasnfer operator given by
\begin{equation}\label{the definition of the transfer operator}
\mathcal{L}_t f(x, {u}):=\sum_{y \in \sigma^{-1}x} g(y)\left\|\mathcal{A}(y)^* \frac{u}{\|u\|}\right\|^t f\left(y, \overline{\mathcal{A}(y)^* u}\right),
\end{equation}
acting on the space of H\"older continuous functions $f: \Sigma^+ \times \P \rightarrow \mathbb{R}$.  
\end{defn}

It is clear from the definition that
$$
\mathcal{L}_t^n f(x,{u})=\sum_{y \in \sigma^{-n} x} g^{(n)}(y)\left\|\mathcal{A}^{[n]}(y) \frac{u}{\|u\|}\right\|^t f\left(y, \overline{\mathcal{A}^{[n]}(y) u}\right)
$$
for every H\"older continuous function $f$ and every $n\ge 1$.

\medskip
The next result, due to Park and Piraino \cite{Park-Piraino}, describes the spectral properties of the transfer operator 
$\mathcal{L}_t$ acting on the space of H\"older continuous functions for all sufficiently small values of $t$, by establishing a spectral gap for the operator $\mathcal L_0$ and by using perturbation theory. More precisely:


\begin{prop}[{\cite[Proposition 9 and Lemma 6]{Park-Piraino}}]\label{main-prop} Suppose that $\mathcal{A}$ is 1-typical fiber-bunched cocycle and $\alpha>0$ is sufficiently small. There exists $t_0>0$ 
such that for any $t \in (-t_0,t_0)$ the operator $\mathcal{L}_t: C^\alpha\left(\Sigma^+ \times \P\right) \rightarrow C^\alpha\left(\Sigma^+ \times \P\right)$ can be written as
$$
\mathcal{L}_t=\rho_t\left(P_t+S_t\right),
$$
where $\rho_t >0$ is the leading eigenvalue and spectral radius of $\mathcal{L}_t$,
$P_t$ is the projection onto the one-dimensional subspace $\operatorname{ker}\left(\rho_t I-\mathcal{L}_t\right)$, and there exist constants $C>0$ and $0<\beta<1$ such that 
$ 
\left\|S_t^n\right\| \leq C \beta^n
$ 
for all $t \in (-t_0,t_0)$ and all $n\ge 1$. Moreover, 
the operators $P_t$ and $S_t$  
satisfy  $P_t S_t=S_t P_t=0$.
Furthermore, the functions $(-t_0,t_0) \ni t \mapsto \rho_t, P_t, S_t$ are real analytic.
\end{prop}

\begin{rem}
 One can compare the class of transfer operators ~\eqref{the definition of the transfer operator} with the ones studied by Le Page \cite{lepage1982theoremes} in the context of strongly irreducible, proximal, i.i.d. random walks satisfying a
finite exponential moment condition (see also \cite[Theorem 4.3]{baladi2000positive}). As the random product of matrices is coded by a one-step cocycle, the family of transfer operators in \cite{lepage1982theoremes} acts on the space of H\"older continuous functions $C^\alpha(\P)$, whereas our class of transfer operator acts on $C^\alpha(\Sigma^+ \times \P)$. 
\end{rem}


We will need an additional description of the spectral properties of the dual transfer operators $\L_t^*$.  By Proposition \ref{main-prop} and as a straightforward application of the geometric Hahn–Banach theorem (see, e.g., \cite[Lemma 4.3]{VV}), there exists a measure $\hat \nu_t$ on $\Sigma^+ \times \P$ such that 
$$
\mathcal{L}_t^{\ast} \hat \nu_t = \rho_t \hat \nu_t
$$
for $t \in (-t_0, t_0)$. Additionally, it follows from the proof of Proposition \ref{main-prop} that 
there exists a function ${\hat h_t} \in C^\alpha\left(\Sigma^+ \times \P\right)$ 
so that 
$\mathcal{L}_t {\hat h_t}=\rho_t {\hat h_t}$ and $P_t {\hat h_t}={\hat h_t}$. This also implies that $P_t^2=P_t$ and $P_t \mathcal{L}_t=\mathcal{L}_t P_t=\rho_t P_t$. 
By writing $P_t f =\kappa_f {\hat h_t}$ for any $f\in C^{\alpha}(\Sigma^+ \times \P)$ and using the Riesz-Markov representation theorem and that $P_t {\hat h_t}={\hat h_t}$ it is not hard to check that 
\begin{equation}\label{eq2}
  P_t f=\left(\int f \mathrm{~d} \hat \nu_t\right) {\hat h_t}, \quad \forall f \in C^{\alpha}(\Sigma^+ \times \P).
\end{equation}

\begin{cor}\label{corollary} Under the assumptions of Proposition \ref{main-prop}, there exist constants $C_t>0$ and  $0<\beta<1$ (as in Proposition \ref{main-prop}) such that,
for all $n \geq 0$ and $f \in C^{\alpha}\left(\Sigma \times \P\right)$,
$$
\left\|\rho_t^{-n}\L_t^n f-\left\langle f, \hat \nu_t\right\rangle {\hat h_t}\right\|_{\alpha} \leq C\|f\|_{\alpha} \beta^n.
$$
\end{cor}
\begin{proof}
  By Proposition \ref{main-prop},  $\rho_t^{-n}\L_t^n=P_t+S_{t}^n.$ Now, the characterization \eqref{eq2} guarantees that
$\left\|\rho_t^{-n} \L_t^n f-\left\langle f, \hat \nu_t\right\rangle {\hat h_t}\right\|_{\alpha} \leq \|S_t^n f\|_{\alpha} \leq \|S_t^n \| \|f\|_{\alpha}
\le C\|f\|_{\alpha} \beta^n,$ 
thus proving the corollary.
\end{proof}

\begin{lem}\label{h is bdd} 
There exists $0<t_1< t_0$ so that
$$
\inf_{t \in (-t_1,t_1)} \inf _{(x, \bar{u}) \in \Sigma^+ \times \P} {\hat h_t}(x, \bar{u})>0
$$
and
$$
\sup_{t \in (-t_1,t_1)} \sup _{(x, \bar{u}) \in \Sigma^+ \times \P} {\hat h_t}(x, \bar{u})<\infty .
$$
\end{lem}
\begin{proof}
Assume for contradiction that the first inequality fails for every small $t_1>0$. Then there exists a sequence $t_n \to 0$ and points $(x_n, \overline{u_n}) \in \Sigma^+ \times \P$ such that $\hat h_{t_n}(x_n, \overline{u_n}) < \frac{1}{n}$. By compactness of $\Sigma^+ \times \P$, the sequence $(x_n, \overline{u_n})$ has a convergent subsequence to a point $(x, \bar{u}) \in \Sigma \times \P$. Without loss of generality, we assume $(x_n, \overline{u_n})_n$ is convergent to $(x, \bar{u})$.
Let $\delta_{(x, \bar{u})}$ denote the Dirac measure at $(x, \bar{u})$. Using ~\eqref{eq2} and the analiticity of the map $t \mapsto {\hat h_t} = P_t 1$, we obtain
$$
{\hat h_t}(x, \bar{u})=\int {\hat h_t} \mathrm{~d} \delta_{(x, \bar{u})} \longrightarrow \int {\hat h_0} \mathrm{~d} \delta_{(x, \bar{u})}={\hat h_0}(x, \bar{u})=1 \quad \text { as } t \rightarrow 0.
$$
The latter implies that $\hat h_{t_n}(x_n, \overline{u_n}) \to 1$, contradicting the assumption that $\hat h_{t_n}(x_n, \overline{u_n}) < \frac{1}{n}$. This proves the first inequality. 
The proof of the second inequality follows by a similar argument. 
\end{proof}

\subsection{Disintegration and support of the eigenmeasures for $\L_t$}

We proceed to show that the probability measures $\hat \nu_t$ satisfying $\mathcal{L}_t^{\ast} \hat \nu_t = \rho_t \hat \nu_t$ are $H^u$-invariant. For that reason, let us introduce the the skew product $\hat{F}_{\mathcal{A}_*^{-1}}: \Sigma^{+} \times \P \rightarrow \Sigma^{+} \times \P$ defined by
$$
\hat{F}_{\mathcal{A}_*^{-1}}(x, \bar{u})=\left(\sigma x, \overline{\mathcal{A}_*^{-1}(x) u}\right),
$$
which is related to the cocycle $\mathcal{A}_*^{-1}$.
As mentioned in Subsection \ref{subsec:fiberbunching}, there is no loss of generality in studying linear cocycles $\mathcal{A}: \Sigma \rightarrow \mathrm{GL}_d(\mathbb{R})$ that are constant along the local stable sets. 
For such a cocycle, the unstable holonomy will be denoted by $H^{u}$ whereas $H^s \equiv I$.

\begin{lem}\label{invariant measure}
Let $\hat{\nu}_t$ be a probability measure on $\Sigma^+ \times \P$ such that $\L_t^{\ast} \hat{\nu}_t =\rho_t \hat{\nu}_t$. Assume that the probability measure $\hat m_t$ on $\Sigma^+ \times \P$ is absolutely continuous with respect to $\hat \nu_t$ with density $d \hat m_t/d\hat \nu_t={\hat h_t}$. Then, $\hat m_t$ 
 is $\hat{F}_{\A_{\ast}^{-1}}$-invariant, and there exists a unique ${F}_{\mathcal{A}_*^{-1}}$-invariant probability measure $m_t$   so that $\pi^+_*m_t=\hat m_t$.
 Additionally, if the $\sigma$-invariant probability measure  $\pi_*m_t$ has local product structure then $m_{t,x}=\delta_{\xi_{\ast}(x)}$ for $\widehat{\pi_*m_t}$-a.e. $x$. In particular, $m_t$ is unique 
and given by the expression
$$
m_t = \int_{\Sigma \times \P} {\hat h_t}\, d\hat\nu_t = \int_\Sigma \delta_{\xi_{\ast}(x)} \; d\widehat{\pi_*m_t}(x)
$$
\end{lem}
\begin{proof}


First, we show that $\hat m_t$ is an $\hat{F}_{\A_{\ast}^{-1}}$ invariant. 
Note that
$ 
\left(\mathcal{L}_t f_1\right) \cdot f_2=\mathcal{L}_t\left(f_1 \cdot f_2 \circ \hat{F}_{\mathcal{A}_*^{-1}}\right)
$ 
for every continuous functions $f_1, f_2 \in C\left(\Sigma^{+} \times \P\right)$,
where $\mathcal{A}_*^{-1}$ is the adjoint of the inverse cocycle for $\mathcal{A}$.  This implies that
$$\begin{aligned}
\hat m_t(f)=\hat{\nu_t}\left({\hat h_t} f\right) 
& =\hat{\nu_t}\left(\rho_t^{-1} \mathcal{L}_t {\hat h_t} \cdot f\right) \\
& =\rho_t^{-1} \hat{\nu_t}\left(\mathcal{L}_t\left({\hat h_t} \cdot(f \circ\hat{F}_{\mathcal{A}_*^{-1}})\right)\right) \\
& =\hat{\nu_t}\left({\hat h_t} \cdot(f \circ \hat{F}_{\mathcal{A}_*^{-1}})\right)
= \hat m_t\left((f \circ \hat{F}_{\mathcal{A}_*^{-1}}\right),
\end{aligned}
$$
which shows that  $\hat m_t$ is $\hat{F}_{\mathcal{A}_*^{-1}}$-invariant. 
As ${F}_{\mathcal{A}_*^{-1}}: \Sigma \times \P \to \Sigma \times \P$ is an extension of $\hat{F}_{\mathcal{A}_*^{-1}}$ and $(\pi^+\times id)^{-1}(x,u)=W^s_{\loc}(x)\times\{u\}$
then it well known that there exists a unique ${F}_{\mathcal{A}_*^{-1}}$-invariant probability measure $m_t$ so that $\pi^+_*m_t=\hat m_t$.
\color{black}

Now, we prove uniqueness.  
Let $(\hat m_{t,x})_{x\in \Sigma^+}$ be the disintegration of $\hat m_t$ on the measurable partition $(\{x\} \times \P)_{x\in \Sigma^+}$.
By item (4) in Theorem~\ref{BV's results} equation (applied to $\A_{\ast}^{-1}$),
\[
m_{t,x} = \lim_{n \to \infty} (\A_{\ast}^{-1})^n (\Pi(\sigma^{-n} x)) \hat m_{t,\Pi(\sigma^{-n} x)} = \lim_{n \to \infty} \hat m_{t,x} = \hat m_{t,x}.
\]
This implies that $\hat m_{t,x} = \hat m_{t,y}$ for any $y\in W^{s}_{\text{loc}}(x)$ and for $\widehat{\pi_*m_t}$-a.e. $x$. Since the local stable holonomies for $\A_{\ast}^{-1}$ are identically equal to the identity $I$ (due to the same reasoning which shows that $H^s \equiv I$ for $\A$, see \eqref{stable holonomy}), it follows that $\nu$ is $(\A_{\ast}^{-1}, H^s)$-invariant. This proves the uniqueness. Indeed, this is because the $\left(\mathcal{A}_*^{-1}, H^s\right)$-invariant measure is unique when $\mathcal{A}$ (hence $\mathcal{A}_*^{-1}$ ) is 1-typical, hence the last statement in the lemma follows as a consequence of item (4) in Theorem \ref{BV's results}. 
This proves the lemma.
\end{proof}

\begin{rem}\label{consequnece of the definition}
Note that 
$\int f(x, \bar{u}) \mathrm{d} m_t(x, \bar{u})=\int f\left(x, \overline{\xi_*}(x)\right) d\widehat{\pi_*m_t}(x)$ 
for every integrable function $f$,
as a consequence of Lemma~\ref{invariant measure}.
Moreover, when $t=0$ it is simple to check that 
${\hat h_0}\equiv 1$ and that $\pi_*\hat m_0 =\pi_*\hat \nu_0=\hat \mu$ is the probability measure given in Subsection~\ref{subsec-psi}.
\end{rem}

\begin{cor}\label{not supported on a projective subspace}
If the $\sigma$-invariant probability measure  $\pi_*m_t$ has local product structure then $\hat{\nu_t}$ is not supported on a proper projective subspace.
\end{cor}
\begin{proof}
Notice that $\hat{m}_t\ll \hat \nu_t$ is $\hat{F}_{\A^{-1}_{\ast}}$-invariant and that the density is H\"older continuous, hence the corollary is an immediate consequence of
item (3) in Theorem \ref{BV's results}.
\end{proof}

\section{Gibbs measures and equilibrium states}
\label{sec:Gibbs-eq-states}

In this section we prove Theorem~\ref{thm:Main1}. More precisely, we construct $\sigma$-invariant Gibbs measures, which are unique and equilibrium states for the potentials $\Phi_{\A,t}$. Such $\sigma$-invariant probability measures ${\hat \mu_t}$, constructed in Subsection~\ref{subsec:defmut}
from the leading eigenvalues and eigenfunction for the operators $\L_t$, are shown to be mixing.

\subsection{Shift invariant measures}
\label{subsec:defmut}
The construction of equilibrium states in the classical thermodynamical formalism gathers both leading eigenfunction and eigenmeasure (cf.\cite{Bow}). 
This suggests to consider the probability measures $\hat m_t$ on the space $\Sigma^+\times \P$ and ${\hat \mu_t}$ on the space $\Sigma^+$ by
\begin{equation}\label{definition of our measure mt}
\int f(x,\bar u) \mathrm{d}  \hat m_t(x,\bar u):=\int f(x,\bar u) {\hat h_t}(x, \bar{u}) \mathrm{d} \hat{\nu}_t(x, \bar{u}), \quad \forall f \in C^\alpha\left(\Sigma^+ \times \P\right),
\end{equation}
and 
\begin{equation}\label{definition of our measure}
\int f(x) \mathrm{d}  {\hat \mu_t}(x):=\int f(x) {\hat h_t}(x, \bar{u}) \mathrm{d} \hat{\nu}_t(x, \bar{u}), \quad \forall f \in C^\alpha\left(\Sigma^+\right),
\end{equation}
respectively.
\begin{rem}\label{rem: projected measure}
Using cylinder sets generate the sigma-algebra on $\Sigma^+$ and the characteristic functions $\{1_{[I]} \colon I \in \Sigma^*\}$ are H\"older continuous it is clear from the definition that ${\hat \mu_t}=\pi_* \hat m_t$ for every $-t_1\le t\le t_1$.    
\end{rem}

The next lemma will be instrumental to describe the measure of cylinder sets.

\begin{lem}\label{definition of Gibbs}
Let ${\hat \mu_t}$ be defined by \eqref{definition of our measure}.
For any $I \in \Sigma^*$, we have
$$
\begin{aligned}
{\hat \mu_t}([I])=&\frac{1}{\rho_t^{|I|}}\int \L_t^{|I|}(1_{([I]}{\hat h_t})(x, \bar{u})d\hat{\nu}_t(x,\bar{u})\\
&=\small{\frac{1}{\rho_t^{|I|}}\int \sum_{y \in \sigma^{-|I|}x} g^{(|I|)}(y)\left\|\mathcal{A}^{[|I|]}(y) \frac{u}{\|u\|}\right\|^t 1_{[I]}(y) {\hat h_t}\left(y, \overline{\mathcal{A}^{[|I|]}(y) u}\right) \mathrm{d} \hat{\nu}_t(x, \bar{u})}
\end{aligned}
$$
for every $|t|\le t_0$. Moreover, for each $-t_1\le t\le t_1$ there exists $C_0>0$ and $R_t>0$ satisfying $\lim_{t\to 0} R_t=1$ and so that 
\begin{equation}
    \label{eq:comparisonmut}
    C_0^{-1} R_t^{-|I|} \hat \mu([I]) \le \hat \mu_t([I]) \le C_0 R_t^{|I|} \hat \mu([I]),\quad\text{for every $I\in \Sigma^*$.}
\end{equation}
\end{lem}

\begin{proof}
 Recalling that $\hat{\nu}_t$  is the eigenmeasure for $\mathcal{L}_t$ and that $1_{[I]}$ is H\"older continuous for each  $I \in \Sigma^*$  we have
$$
\begin{aligned}
{\hat \mu_t}([I]) & =\int 1_{[I]}(x) {\hat h_t}(x, \bar{u}) \mathrm{d} \hat{\nu}_t(x, \bar{u}) 
=\frac{1}{\rho_t^{|I|}} \int  \L_t^{|I|}(1_{([I]}{\hat h_t})(x, \bar{u})d\hat{\nu}_t(x,\bar{u})
\end{aligned}
$$
which proves the first part of the lemma.  The rest of the proof of this lemma is inspired by \cite{Rush}. Now, if $0\le t\le t_1$ then using Lemma~\ref{h is bdd},
$$
\begin{aligned}
{\hat \mu_t}([I])
&=\small{\frac{1}{\rho_t^{|I|}}\int \sum_{y \in \sigma^{-|I|}x} g^{(|I|)}(y)\left\|\mathcal{A}^{[|I|]}(y) \frac{u}{\|u\|}\right\|^t 1_{[I]}(y) {\hat h_t}\left(y, \overline{\mathcal{A}^{[|I|]}(y) u}\right) \mathrm{d} \hat{\nu}_t(x, \bar{u})}
\\
&\le \small{\frac{1}{\rho_t^{|I|}}\int \sum_{y \in \sigma^{-|I|}x} g^{(|I|)}(y)\left\|\mathcal{A}^{[|I|]}(y) \frac{u}{\|u\|}\right\|^t 1_{[I]}(y) {\hat h_t}\left(y, \overline{\mathcal{A}^{[|I|]}(y) u}\right) \mathrm{d} \hat{\nu}_t(x, \bar{u})}
\\ 
&  \le C \small{\frac{1}{\rho_t^{|I|}} \sup_{|t\le t_1|} \|\hat h_t\|_\infty 
\cdot \max_{x\in [I]}\|\A^{[|I|]}(x)\|^{t}
\int g^{(|I|)}(Ix)  \mathrm{d} \hat{\nu}_t(x, \bar{u})} \\
& \le C \sup_{|t\le t_1|} \|\hat h_t\|_\infty 
\cdot \Big(\frac{\|\A\|_\infty^{t}}{\rho_t}\Big)^{|I|} \cdot {\hat \mu}_0 ([I]).
\end{aligned}
$$
Conversely, 
$$
\begin{aligned}
{\hat \mu_t}([I])
&=\small{\frac{1}{\rho_t^{|I|}}\int \sum_{y \in \sigma^{-|I|}x} g^{(|I|)}(y)\left\|\mathcal{A}^{[|I|]}(y) \frac{u}{\|u\|}\right\|^t 1_{[I]}(y) {\hat h_t}\left(y, \overline{\mathcal{A}^{[|I|]}(y) u}\right) \mathrm{d} \hat{\nu}_t(x, \bar{u})}
\\
&\ge \small{\frac{1}{\rho_t^{|I|}}\int \sum_{y \in \sigma^{-|I|}x} g^{(|I|)}(y)\left\|\mathcal{A}^{[|I|]}(y) \frac{u}{\|u\|}\right\|^t 1_{[I]}(y) {\hat h_t}\left(y, \overline{\mathcal{A}^{[|I|]}(y) u}\right) \mathrm{d} \hat{\nu}_t(x, \bar{u})}
\\ 
&  \ge C \small{\frac{1}{\rho_t^{|I|}} \inf_{(t,x,\bar u)\in [0,t_1] \times \Sigma^+ \times \P} |\hat h_t(x,u)|
\cdot \min_{x\in [I]}\|\A(x)^{-1}\|^{-t|I|}
\int g^{(|I|)}(Ix)  \mathrm{d} \hat{\nu}_t(x, \bar{u})} \\
& = C \inf_{(t,x,\bar u)\in [0,t_1] \times \Sigma^+ \times \P} |\hat h_t(x,u)|
\cdot 
 \Big(\frac{\min_{x\in [I]}\|\A(x)^{-1}\|^{-t}}{\rho_t}\Big)^{|I|} \cdot {\hat \mu}_0 ([I])
\end{aligned}
$$
This proves the lemma in case of positive $t$.
The proof for $-t_1\le t<0$ is identical, hence omitted. 
\end{proof}

We proceed to show that the measures ${\hat \mu_t}$ are $\sigma$-invariant and ergodic. Observe that any function $f \in C(\Sigma^+)$ can naturally be regarded as a function on $\Sigma^+ \times \P$ by defining $f(x, \bar{u}) := f(x)$, for some fixed $\bar u\in \P$.

\begin{lem}
${\hat \mu_t}$ is a $\sigma$-invariant probability measure.
\end{lem}

\begin{proof}
Our proof follows an adaptation of \cite[Lemma 1.13]{Bow}. For $f \in C^{\alpha}\left(\Sigma^{+}\right)$,
$$\begin{aligned}
\left(\left(\mathcal{L}_t {\hat h_t}\right) \cdot f\right)(x, \bar{u}) & =\sum_{y \in \sigma^{-1} x} g(y)\left\|\mathcal{A}(y)^* \frac{\bar u}{\|\bar u\|}\right\|^t {\hat h_t}\left(y, \overline{\mathcal{A}(y)^* \bar u}\right) f(x) \\
& =\sum_{y \in \sigma^{-1} x} g(y)\left\|\mathcal{A}(y)^* \frac{\bar u}{\|\bar u\|}\right\|^t {\hat h_t}\left(y, \overline{\mathcal{A}(y)^* \bar u}\right) f(\sigma y)\\
&=\mathcal{L}_t\left({\hat h_t} \cdot(f \circ \sigma)\right)(x, \bar{u}).
\end{aligned}
$$
Thus, given an arbitrary $f \in C^{\alpha}\left(\Sigma^{+}\right)$,
$$\begin{aligned}
{\hat \mu_t}(f) & =\hat{\nu}_t\left({\hat h_t} f\right) \\
& =\hat{\nu}_t\left(\rho_t^{-1} \mathcal{L}_t {\hat h_t} \cdot f\right) \\
& =\rho_t^{-1} \hat{\nu}_t\left(\mathcal{L}_t\left({\hat h_t} \cdot(f \circ \sigma)\right)\right) \\
& =\hat{\nu}_t\left({\hat h_t} \cdot(f \circ \sigma)\right)  ={\hat \mu_t}(f \circ \sigma) .
\end{aligned}
$$
\end{proof}

\begin{lem} The measure ${\hat \mu_t}$ is mixing (hence ergodic).
\end{lem}
\begin{proof}
This follows by an argument similar to the proof of \cite[Proposition 1.14]{Bow}. We proceed to show that, for any cylinder sets $E, F \subset \Sigma^{+}$,
$ 
\lim _{n \rightarrow \infty} {\hat \mu_t}\left(E \cap \sigma^{-n} F\right)={\hat \mu_t}(E) {\hat \mu_t}(F)
$ 
(the extension for arbitrary Borel sets $E,F$ follows a standard argument).
For any $f_1, f_2 \in C^{\alpha}\left(\Sigma^{+}\right)$ and $n\ge 1$,
$$
\begin{aligned}
\left(\mathcal{L}_t^n\left({\hat h_t} f_1\right) \cdot f_2\right)(x, \bar{u}) & =\sum_{y \in \sigma^{-n_x}} g(y)\left\|\mathcal{A}(y)^* \frac{u}{\|u\|}\right\|^t {\hat h_t}\left(y, \overline{\mathcal{A}(y)^* u}\right) f_1(y) f_2(x) \\
& =\sum_{y \in \sigma^{-n} x} g(y)\left\|\mathcal{A}(y)^* \frac{u}{\|u\|}\right\|^t {\hat h_t}\left(y, \overline{\mathcal{A}(y)^* u}\right) f_1(y) f_2\left(\sigma^n y\right) \\
& =\mathcal{L}_t^n\left({\hat h_t} f_1 \cdot\left(f_2 \circ \sigma^n\right)\right)(x, \bar{u}).
\end{aligned}
$$
Thus, for $E=\left[x\right]_k$ and $F=\left[y \right]_{\ell}$, 
$$
\begin{aligned}
{\hat \mu_t}\left(E \cap \sigma^{-n} F\right) & ={\hat \mu_t}\left(1_E \cdot 1_{\sigma^{-n}(F)}\right) \\
& ={\hat \mu_t}\left(1_E \cdot\left(1_F \circ \sigma^n\right)\right) 
 =\nu_t\left({\hat h_t} 1_E \cdot\left(1_F \circ \sigma^n\right)\right) \\
& =\rho_t^{-n} \mathcal{L}_t^{* n} \hat{\nu}_t\left({\hat h_t} 1_E \cdot\left(1_F \circ \sigma^n\right)\right) 
 =\hat{\nu}_t\left(\rho_t^{-n} \mathcal{L}_t^n\left({\hat h_t} 1_E \cdot\left(1_F \circ \sigma^n\right)\right)\right. \\
& =\hat{\nu}_t\left(\rho_t^{-n} \mathcal{L}_t^n\left({\hat h_t} 1_E\right) \cdot 1_F\right) 
\end{aligned}
$$
for every $n\ge 1$.
Therefore, ${\hat h_t} 1_E \in C^\alpha\left(\Sigma^{+} \times \P\right)$, Corollary \ref{corollary} yields
$$
\begin{aligned}
\left.\mid {\hat \mu_t}(E) \cap \sigma^{-n} F\right)-{\hat \mu_t}(E) {\hat \mu_t}(F) \mid & =\left|\hat{\nu}_t\left(\rho_t^{-n} \mathcal{L}_t^n\left({\hat h_t} 1_E\right) \cdot 1_F\right)-\hat{\nu}_t\left({\hat h_t} 1_E\right) \hat{\nu}_t\left({\hat h_t} 1_F\right)\right| \\
& =\left|\hat{\nu}_t\left(\left(\rho_t^{-n} \mathcal{L}_t^n\left({\hat h_t} 1_E\right)-\hat{\nu}_t\left({\hat h_t} 1_E\right) {\hat h_t}\right) 1_F\right)\right| \\
& \leq\left\|\rho_t^{-n} \mathcal{L}_t^n\left({\hat h_t} 1_E\right)-\hat{\nu}_t\left({\hat h_t} 1_E\right) {\hat h_t}\right\|_{\infty} \hat{\nu}_t\left(1_F\right), 
\end{aligned}
$$
which tends to zero as $n$ tends to infinity.
\end{proof}






\subsection{Large deviations}
Throughout this section we will always assume that $\A$ is a 1-typical fiber-bunched cocycle.
Let $({\hat \mu_t})_{-t_1\le t\le t_1}$ be the family of probability measures defined in the previous subsection. Recall that $\hat \mu_0=\hat \mu$ is the $\sigma$-invariant probability measure defined in Subsection \ref{subsec-psi} (cf. Remark~\ref{consequnece of the definition}).

\medskip
We will collect some results that provide large deviations estimates for linear cocycles.
Gou\"ezel and Stoyanov proved exponential large deviations for all Lyapunov exponents in case of 1-typical linear cocycles over the shift and Gibbs measures for H\"older continuous potentials
(cf. \cite[Theorem 1.5 (5)]{GouezelStoyanov2019}). In particular,
for every small enough $\ep>0$ there exists $C:=C(\ep)>0$ such that
\begin{equation}
\label{LD}
\hat{\mu}\left\{x\in \Sigma^+:\left|\log \left\|\mathcal{A}^{[n]}(x)\right\|-n \lambda_1(\hat{\mu}, \A)\right|\ge n \varepsilon\right\} \le C e^{-C^{-1}n},    
\end{equation}
for every $n\ge 1$.
Furthermore, Rush \cite{Rush} used Gartner-Ellis theorem to prove the following large deviations principle for the probability measures $\hat \nu_t$:
there exists $\delta_0>0$ and for every  $\vep>0$ there exists $\Lambda_\vep>0$ so that 
\begin{equation}
    \label{eq:Rushnut}
\limsup_{n\to\infty} \frac1n \log \hat\nu_t
    \Big((x,\bar u)\in \Sigma^+\times \P \colon \Big|\log \|\A^{[n]}(x)\bar u\| - n \lambda_1(\hat \mu_t,\A) \Big|>\vep \Big)
    \le -\Lambda_\vep
\end{equation}
for every $|t|\le \delta_0$ (even though the main results in \cite{Rush} hold strictly for locally constant cocycles, the proof of Proposition 5.1, Corollary~5.2 and Proposition~5.6 in \cite{Rush} explore the transfer operators techniques and hold for more general H\"older continuous cocycles). 
We recall $\overline{\xi_*}(x)$ is the slowest Oseledets' subspace of $\mathcal{A}_*^{-1}$ at $x$. We will need the following auxiliary lemma:

\begin{lem}
\label{le:57rush}
There exists $0<\delta_1<\delta_0$ so that, for each $t\in (-\delta_1,\delta_1)$,
$$ 
\hat\nu_t\left((x,u)\in \Sigma^+ \times \P \colon (x, \bar{u})\neq(x, \overline{\xi_*}(x)) \; \right)=0.
$$ 
\end{lem}

\begin{proof}
The lemma is a direct consequence of Proposition 5.7 in \cite{Rush}, whose proof relies on the large deviation estimates in \eqref{eq:Rushnut} and holds for non-locally constant cocycles.
\end{proof}

\begin{lem}\label{LD application}
    For every $\ep>0$ there exist $\delta>0$, $C_\vep,\Lambda_\vep>0$ and $N>1$ so that, for every $-\delta\le t\le \delta$:
    \begin{itemize}
    \item[(1)] 
    $
    \hat{\mu}_t\left\{x\in \Sigma^+:
    \left|\frac{1}{n} \log \left\|\mathcal{A}^{[n]}(x) \xi_*\left(\sigma^n x\right)\right\|-\lambda_1(\hat\mu,\A)\right|
    \ge \varepsilon\right\} \le C_\vep e^{-\Lambda_{\vep}n},
    $
\item[(2)] 
$
    \hat{\mu}_t\left\{x\in \Sigma^+:
    \left|\frac{1}{n} \log \left\|\mathcal{A}^{[n]}(x)\right\|-\lambda_1(\hat\mu,\A)\right|
    \ge \varepsilon\right\} \le C_\vep e^{-\Lambda_{\vep}n},
    $
\item[(3)] 
$
\hat{\mu}_t\left\{x\in \Sigma^+:
    \left|\frac{1}{n} \log \gamma_{1,2}\left(\mathcal{A}^{[n]}(x)\right)-\left(\lambda_2(\hat\mu,\A)-\lambda_1(\hat\mu,\A)\right)\right|
    \ge \varepsilon\right\} \le C_\vep e^{-\Lambda_{\vep}n},
$
and
\item[(4)] 
$
\sup_{\bar{v} \in \P} \, \hat{\mu}_t\left\{x\in \Sigma^+:
\left|\frac{1}{n} \log \left\|\mathcal{A}^n(x) \frac{v}{\|v\|}\right\|-\lambda_1(\hat\mu,\A)\right|
    \ge \varepsilon\right\} \le C_\vep e^{-\Lambda_{\vep}n},
$
\end{itemize}
for every $n\ge N$.
Moreover, $\lim_{\vep\to 0} \Lambda_\vep=\Lambda_0>0$.
\end{lem}

\begin{proof}
The proof uses the relation between $\hat\mu_t$ and $\hat\mu$ from Lemma~\ref{definition of Gibbs}, together with the large deviation estimates \eqref{LD}, Lemma~\ref{le:57rush}, and equation~\eqref{eq:Rushnut}, in a manner identical to the proof of \cite[Lemma~6.7]{Rush}, and is therefore omitted.

\end{proof}

\begin{lem}\label{LD application2}
   For every $\ep>0$ there exist $\delta>0$,  $C_\vep,\Lambda_\vep>0$ and $N\ge 1$ such that
 \begin{itemize}
\item[(1)] $ \hat{\mu}_t\left\{x \in \Sigma^+: d_{\P}\left(\overline{\xi_*}(x), \overline{v_{+}}\left(\mathcal{A}^{[n]}(x)\right)\right)>e^{-\left(\lambda_1(\hat\mu,\A)-\lambda_2(\hat\mu,\A)-\varepsilon\right) n}\right\} \leq C_\vep e^{-\Lambda_{\vep}n}$,
\item[(2)] $ \sup_{\bar{v} \in \P} \hat{\mu}_t\left\{x \in \Sigma^+: \delta_{\P}\left(\overline{v_{+}}\left(\mathcal{A}^{[n]}(x)\right), \bar{v}\right)<e^{-\frac{\varepsilon}{2} n}\right\} \leq C_\vep e^{-\Lambda_{\vep}n}$,
\item[(3)] $\sup_{\bar{v} \in \P} \hat{\mu}_t\left\{x \in \Sigma^+: \delta_{\P}\left(\overline{\xi_*}(x), \bar{v}\right)<e^{-\varepsilon n}\right\} \leq C_\vep e^{-\Lambda_{\vep}n}$
 \end{itemize}
 for every $n\ge N$ and $|t|\le \delta$.
 Moreover $\lim_{\vep\to 0} \Lambda_\vep=\Lambda_0>0$.
\end{lem}

\begin{proof} 
The proof uses the relation between $\hat\mu_t$ and $\hat\mu$ in Lemma~\ref{definition of Gibbs} 
together with
Lemma~\ref{LD application} and the argument is identical to the proof of \cite[Lemma~6.8]{Rush}, hence omitted.
\end{proof}

The following lemma will be used to show that the upper bound of the Gibbs property holds for the measures ${\hat \mu_t}$ when $t$ is negative, and the lower bound of the Gibbs property holds for the measures ${\hat \mu_t}$ when $t$ is positive. Recall that $\Lambda_0$ 
is given by the large deviation property. 

\begin{lem}\label{inequlity for the neqative s} 
Let  $0<s_1<\Lambda_0$ and $s_2>0$. Then, there is $\delta_{1}^{'}>0$ such that for all $|t| \leq \delta_1^{'}$,
there exists $C_1, C_2>0$ (depending on $s_1$ and $s_2$, respectively) 
such that for any $I \in \Sigma^*$
we have
$$
\int\left\|\A^{[|I|]} (Ix) \frac{u}{\|u\|}\right\|^{-s_1} \mathrm{~d} \hat{\nu}_t(x, \bar{u}) 
\leq C_1 \|\A(I)\|^{-s_1}
$$
and
$$
\int\left\|\A^{[|I|]} (Ix) \frac{u}{\|u\|}\right\|^{s_2} \mathrm{~d} \hat{\nu}_t(x, \bar{u}) 
\geq C_2 \|\A(I)\|^{s_2}.
$$
\end{lem}

\begin{proof}    
 Let $\Lambda_0>0$ be given \eqref{LD} and
give $\vep>0$ small so that $s_1\vep < \frac{\Lambda_0}4$ and $\Lambda_\vep>\frac{\Lambda_0}2$. Note that Lemma~\ref{LD application2} holds for all $|t|\leq \delta$. Define $\delta_{1}^{'}:=\min\{\delta_1, \delta\}$.
Note that the probability measure $\hat \nu_t$ is equivalent to an ergodic $\hat F_{\A_*^{-1}}$- probability measure $\hat m_t$ (recall \eqref{definition of our measure mt} and Lemmas~\ref{h is bdd} and ~\ref{invariant measure}) and that $\pi_*\hat m_t=\hat\mu_t$ (see Remark \ref{rem: projected measure}).
Hence, by Lemma~\ref{le:57rush} and item (ii) in Lemma~\ref{benoist}, 
$$
\begin{aligned}
&\int\left\|\A^{[|I|]}(Ix) \frac{u}{\|u\|}\right\|^{-s_1} \mathrm{~d} \hat{\nu}_t(x, \bar{u})
 = \int_{[I] \times \P}\left\|\A^{[|I|]}(x) \overline{\xi_*}(x)\right\|^{-s_1} \mathrm{~d} \hat{\nu}_t(x, \bar{u}) \\
& \leq \int_{[I] \times \P} \|\A^{|I|}(x)\|^{-s_1} \; \delta_{\P}\left(\overline{\xi_*}(x), \overline{v_{+}}(\A^{{|I|}}(x))\right)^{-s_1} \mathrm{~d} \hat{\nu}_t(x, \bar{u})\\
& \leq (\inf_{z \in [I]}\|\A^{|I|}(z)\|)^{-s_1} \int_{[I] \times \P}  \delta_{\P}\left(
\overline{\xi_*}(x), \overline{v_{+}}(\A^{{|I|}}(x))\right)^{-s_1} \mathrm{~d} \hat{\nu}_t(x, \bar{u}) \\
& \leq  C \|\A(I)\|^{-s_1} \; \int_{[I] \times \P}  \delta_{\P}\left(\overline{\xi_*}(x), \overline{v_{+}}(\A^{{|I|}}(x))\right)^{-s_1} \; \mathrm{~d} \hat{\nu}_t(x, \bar{u}),
\end{aligned}
$$
where 
$C$ is the constant given by the bounded distortion. 
Given $n\ge 1$, let us consider the set
$$
E_n:=\left\{x \in \Sigma^+ : e^{-\varepsilon(n+1)} \leq \delta_{\P}\left(\overline{\xi_*}(x), \overline{v_{+}}(\A^{{|I|}}(x) \right)<e^{-\varepsilon n}\right\}.
$$
Fix $N\ge 1$. By Tchebychev's inequality, we have 
$$
\begin{aligned}
\int_{[I] \times \P} 
& \delta_{\P}\left(\overline{\xi_*}(x), \overline{v_{+}}(\A^{{|I|}}(x))\right)^{-s_1}
\mathrm{~d} \hat{\nu_t}(x, \bar{u})\\ 
&  \leq e^{s_1 \varepsilon N}+\sum_{n \geq N} \int_{E_n \cap [I]} \delta_{\P}\left(\overline{\xi_*}(x),
\overline{v_{+}}(\A^{{|I|}}(x))
\right)^{-s_1} \mathrm{~d} \hat{\nu_t}(x,\bar u) \\
& 
\leq e^{s_1 \varepsilon N}+ \sum_{n \geq N} e^{\varepsilon s_1 (n+1)} 
\cdot
\hat \nu_t(E_n),
\end{aligned}
$$
\color{black}
which is summable by the choice  $s_1\vep < \frac{\Lambda_0}4$  together
Lemma~\ref{LD application2}. 
Similarly,
$$
\begin{aligned}
&\int\left\|\A^{[|I|]}(Ix) \frac{u}{\|u\|}\right\|^{s_2} \mathrm{~d} \hat{\nu}_t(x, \bar{u})
= \int_{[I] \times \P}\left\|\A^{[|I|]}(x) \overline{\xi_*}(x)\right\|^{s_2} \mathrm{~d} \hat{\nu}_t(x, \bar{u}) \\
& \geq \int_{[I] \times \P} \|\A^{|I|}(x)\|^{s_2} \; \delta_{\P}\left(\overline{\xi_*}(x), \overline{v_{+}}(\A^{{|I|}}(x))\right)^{s_2} \mathrm{~d} \hat{\nu}_t(x, \bar{u})\\
& \geq (\inf_{z \in [I]}\|\A^{|I|}(z)\|)^{s_2} \int_{[I] \times \P}  \delta_{\P}\left(
\overline{\xi_*}(x), \overline{v_{+}}(\A^{{|I|}}(x))\right)^{s_2} \mathrm{~d} \hat{\nu}_t(x, \bar{u}) \\
& \geq  C^{-1} \|\A(I)\|^{s_2} \; \int_{[I] \times \P}  \delta_{\P}\left(\overline{\xi_*}(x), \overline{v_{+}}(\A^{{|I|}}(x))\right)^{s_2} \; \mathrm{~d} \hat{\nu}_t(x, \bar{u})\\
& \geq  C^{-1} \|\A(I)\|^{s_2} \; 
\sum_{n \geq 1} e^{-\varepsilon s_2 (n+1)} \cdot
\hat \nu_t(E_n)
\color{black}
\end{aligned}
$$
which is summable as a consequence of the choice of $\ep$ and Lemma~\ref{LD application2}.
This proves the lemma.
\color{black}

\end{proof}

\subsection{Uniqueness of the Gibbs equilibrium measure}
In this subsection, we will prove that all measures $\hat\mu_t$ satisfy the Gibbs property and use this fact to study the regularity of the pressure function and the Lyapunov exponents. In fact, the differentiability of the map $(0,+\infty) \ni t \mapsto P(\sigma,t\Phi_\A)$ follows from the combination of  \cite[Theorem~4.8]{FH} and \cite{park2020quasi}. Here we prove the following:
\color{black}

\begin{thm}\label{The measure is Gibbs} 
There is $t_2>0$ such that for each  $t\in (-t_2,\delta_1^{'})$ the probability measure ${\hat \mu_t}$ is a $\sigma$-invariant, ergodic Gibbs measure with respect to the  family of potentials $\Phi_{\A, t}$. Moreover,
  the functions
  $$
  (-t_2,\delta_1^{'}) \ni t \mapsto P(\sigma,\Phi_{\A, t})
  \quad\text{and}\quad
  (-t_2,\delta_1^{'}) \ni t \mapsto \lambda_1(\mu_t,\A)=P'(\sigma, \Phi_{\A, t})
  $$
  are real analytic.
\end{thm}

\begin{proof}
Let $t_2:=\min\{\Lambda_0, \delta_1^{'}\}$. The invariance and ergodicity of ${\hat \mu_t}$ follows from 
Subsection~\ref{subsec:defmut}. 
For each $t\in (-t_{2}, \delta_1^{'})$, we claim that $\hat \mu_t$ satisfies the following Gibbs property: 
There exist $C_1,C_2>0$ such that for all $n\ge 1, I \in \Sigma_n$, and $y \in[I]$,
\begin{equation}\label{definition Gibbs with spect}
C_1 \leq \frac{{\hat \mu_t}([I])}{\rho_t^{-n} g^{(n)}(y)\left\|\mathcal{A}^n(y)\right\|^t} \leq C_2.
\end{equation}
\color{black}

\medskip
\noindent\emph{Case 1:} $t\in (-t_2,0]$ 
\medskip

\noindent 
By the bounded distortion for $g$, Lemma \ref{h is bdd}  and Lemma~\ref{definition of Gibbs}, we have that for any $n\ge 1, I \in \Sigma_n$, and $y \in [I]$,
$$
\begin{aligned}
\frac{{\hat \mu_t}([I])}{\rho_t^{-n}} & =\int \sum_{\tilde y \in \sigma^{-n} x} g^{(n)}(\tilde y)\left\|\mathcal{A}^{[n]}(\tilde y) \frac{u}{\|u\|}\right\|^t 1_{[I]}(\tilde y) {\hat h_t}\left(\tilde y, \overline{\mathcal{A}^{[n]}(\tilde y) u}\right) \mathrm{d} \hat{\nu}_t(x, {u}) \\
& =\int  g^{(n)}(Ix)\left\|\mathcal{A}^{[n]}(Ix) \frac{u}{\|u\|}\right\|^t 1_{[I]}(Ix) {\hat h_t}\left(Ix, \overline{\mathcal{A}^{[n]}(Ix) u}\right) \mathrm{d} \hat{\nu}_t(x, {u}) \\
& \leq C \sup \left({\hat h_t}\right)  g^{(n)}(y) \int\left\|\mathcal{A}^{[n]}(Ix) \frac{u}{\|u\|}\right\|^t  \mathrm{d} \hat{\nu}_t(x, {u}) \\
\end{aligned}
$$
where $C$ is the constant given by the bounded distortion for $g$. 
Hence, applying Lemma \ref{inequlity for the neqative s} and using bounded distortion, 
one obtains $C_2=C_2(t)>0$ so that 
$$
\frac{{\hat \mu_t}([I])}{\rho_t^{-n}} \leq C_2 g^{(n)}(y)\left\|\mathcal{A}^n(y)\right\|^t
$$
for every $y \in [I]$.
For the converse inequality, as $t$ is negative we can use that $\left\|A \frac{u}{\|u\|}\right\| \leq\|A\|$ for any $A \in \glr$ and $u \in \mathbb{R}^d$. This, combined with Lemma \ref{h is bdd}  and Lemma~\ref{definition of Gibbs}, ensures that there exists $C_3=C(t)>0$ such that for any $n \in \N$, $I \in \Sigma_n$ and $y \in [I]$,
$$
\begin{aligned}
\frac{{\hat \mu_t}([I])}{\rho_t^{-n}} & =\int \sum_{\tilde y \in \sigma^{-n} x} g^{(n)}(\tilde y)\left\|\mathcal{A}^{[n]}(\tilde y) \frac{u}{\|u\|}\right\|^t 1_{[I]}(\tilde y) {\hat h_t}\left(\tilde y, \overline{\mathcal{A}^{[n]}(\tilde y) u}\right) \mathrm{d} \hat{\nu}_t(x, {u}) \\
& \geq  \inf \left({\hat h_t}\right)   \int g^{(n)}(Ix) \left\|\mathcal{A}^{[n]}(Ix) \right\|^t  \mathrm{d} \hat{\nu}_t(x, {u}) \\
& \geq C \inf \left({\hat h_t}\right)  g^{(n)}(y) \left\|\mathcal{A}^{[n]}(y)\right\|^t\\
& \geq  C_3 g^{(n)}(y)\left\|\mathcal{A}^{[n]}(y)\right\|^t,
\end{aligned}
$$
which completes the proof of the Gibbs property for $t$ negative.

\medskip
\noindent\emph{Case 2:} $t\in (0,\delta_1^{'})$ 
\medskip

\noindent 
By the bounded distortion for $g$, Lemma \ref{h is bdd} and  Lemma~\ref{definition of Gibbs}, we have that for any $n\ge 1, I \in \Sigma_n$, and $y \in [I]$,
$$
\begin{aligned}
\frac{{\hat \mu_t}([I])}{\rho_t^{-n}} & =\int \sum_{\tilde y \in \sigma^{-n} x} g^{(n)}(\tilde y)\left\|\mathcal{A}^{[n]}(\tilde y) \frac{u}{\|u\|}\right\|^t 1_{[I]}(\tilde y) {\hat h_t}\left(\tilde y, \overline{\mathcal{A}^{[n]}(\tilde y) u}\right) \mathrm{d} \hat{\nu}_t(x, {u}) \\
& \leq  \sup \left({\hat h_t}\right)   \int g^{(n)}(Ix) \left\|\mathcal{A}^{[n]}(Ix) \right\|^t  \mathrm{d} \hat{\nu}_t(x, {u}) \\
& \leq C \sup \left({\hat h_t}\right)  g^{(n)}(y) \left\|\mathcal{A}^{[n]}(y)\right\|^t
\end{aligned}
$$
for some constant $C>0$ (by bounded distortion of $g$ and the cocycle $\A$). 
Conversely, by using Lemma~\ref{inequlity for the neqative s}, there is $C'>0$ (depending on the parameter $t$) such that for any $n \in \N$, $I \in \Sigma_n$ and $y \in [I]$, one obtains
$$
\begin{aligned}
\frac{{\hat \mu_t}([I])}{\rho_t^{-n}}  & =\int \sum_{\tilde y \in \sigma^{-n} x} g^{(n)}(\tilde y)\left\|\mathcal{A}^{[n]}(\tilde y) \frac{u}{\|u\|}\right\|^t 1_{[I]}(\tilde y) {\hat h_t}\left(\tilde y, \overline{\mathcal{A}^{[n]}(\tilde y) u}\right) \mathrm{d} \hat{\nu}_t(x, {u}) \\
& \geq  \inf \left({\hat h_t}\right)   \int g^{(n)}(Ix) \left\|\mathcal{A}^{[n]}(Ix) \frac{u}{\|u\|}\right\|^t  \mathrm{d} \hat{\nu}_t(x, {u}) \\
& \geq C^{-1} \inf \left({\hat h_t}\right)  g^{(n)}(y) \int \left\|\mathcal{A}^{[n]}(Ix)\frac{u}{\|u\|}\right\|^t\, d\hat\nu_t(x,u)\\
& \ge 
C' g^{(n)}(y) \|\A^n(y)\|^{t}, 
\end{aligned}
$$
which completes the proof of the Gibbs property for $t$ positive.

Altogether, this proves ~\eqref{definition Gibbs with spect}.

\medskip

\medskip
Let us resume the proof of the theorem. 
 By ~\eqref{definition Gibbs with spect}, ${\hat \mu_t}$ is a Gibbs measure for $\Phi_{\A, t}^g$ and consequently for $\Phi_{\A, t}$ (see \eqref{eq:g-functiondef}). Moreover, 
 \begin{equation}\label{TP is equal to the log SR}
     P(\sigma, \Phi_{\A, t}^g)=\log \rho_t
 \end{equation} by Lemma \ref{Gibbs measure is Eq and the topological pressure}. Therefore, by \eqref{relation between pressures of potentials},
 \begin{equation}\label{relation between the topological pressure}
P\left(\sigma,\Phi_{\A, t} \right)=\log \rho_t+P(\sigma,\psi).
\end{equation}

 

The analyticity of the function $t \mapsto P(\sigma,\Phi_{\A, t})$ on the interval $(-t_2,\delta_1^{'})$ is a consequence of the analyticity of the logarithm of the spectral radius $\rho_t$  (cf. Proposition~\ref{main-prop}). 
Finally, by \eqref{relation between the topological pressure}, the variational principle, and the fact that $\hat{\mu}_t$ is a Gibbs equilibrium measure (see Lemma \ref{Gibbs measure is Eq and the topological pressure}), one obtains,
$$
\lambda(\hat{\mu}_t, \mathcal{A}) = P'(\sigma, \Phi_{\mathcal{A}, t}) = \frac{\rho_t'}{\rho_t},
$$
which varies analytically with respect to $t \in (-t_2, \delta_1^{'})$. This completes the proof of the theorem.
\end{proof}

In the following theorem, we show that the constructed
Gibbs measures ${\hat \mu_t}$ are unique equilibrium states for all negative $t$ (sufficiently close to
0).

\begin{thm}\label{uniqness}
There exists $0<t_3\le t_2$ so that ${\hat \mu_t}$ is the  unique 
equilibrium state for $\Phi_{\A, t}$ 
for  every $t\in (-t_3,+\infty)$. Moreover, $\hat\mu_t$ is a Gibbs measure.
\end{thm}

\begin{proof} 
By Theorem \ref{The measure is Gibbs} and Lemma \ref{Gibbs measure is Eq and the topological pressure}, ${\hat \mu_t}$ is an equilibrium state for $\Phi_{\A, t}$, which has the Gibbs property. Now, we show that it is a unique equilibrium measure.

Given $t \ge 0$ and the fact that $\hat{\mu}_t$ is an equilibrium state for $\Phi_{\mathcal{A}, t}$, it follows from \cite{Feng11, park2020quasi} that $\hat{\mu}_t$ is the unique Gibbs equilibrium state for the potential $\Phi_{\mathcal{A}, t}$ for all $t\geq 0$. Hence, it remains to prove the theorem for negative values of $t$.

\smallskip
 Given $t\in (-t_2,0)$,
we proceed to prove the uniqueness of equilibrium states.
\color{black}

\medskip
\noindent{\bf Claim 1.}
\emph{There exists $0<t_3\le  t_2$ so that, for every $-t_3< t<0$, if $\eta$ is an equilibrium measure for $\Phi_{\A, t}$ then $\lambda_1(\eta, \A) >\lambda_2(\eta, \A)$.}
\medskip

\begin{proof}[Proof of Claim 1]
    Suppose by contradiction this is not the case. Then there exists a sequence $t_n \uparrow 0$ and, for each $n\ge 1$, there exists an equilibrium state  $\eta_{t_n}$ for $\Phi_{\A,t_n}$ such that $\lambda_1\left(\eta_{t_n}, \mathcal{A}\right)=\lambda_2\left(\eta_{t_n}, \mathcal{A}\right)$. Notice that  
$$
h_{\eta_{t_n}}(\sigma)+\int \psi \,\mathrm{d} \eta_{t_n}+t_n \lambda_1\left( \eta_{t_n}, \mathcal{A}\right)=P\left(\sigma,\Phi_{\A, t_n}\right).
$$

Let $\eta^{\prime} \in \M(\sigma)$ be a weak* limit point of the sequence $\left(\eta_{t_n}\right)_{n\ge 1}$ (up to consider a subsequence we assume that $(\eta_{t_n})_n$ converges to $\eta'$). By the upper semi-continuity of the entropy map,
$$
\begin{aligned}
P(\sigma,\psi)&=\lim _{n \rightarrow \infty} P\left(\sigma,\Phi_{\A, t_n}\right)\\
& =\lim_{n \rightarrow \infty} h_{\eta_{t_n}} (\sigma)+\int \psi \,\mathrm{d} \eta_{t_n}+t_n \lambda_1\left(\eta_{t_n},  \mathcal{A} \right) \\
& \leq h_{\eta^{\prime}}( \sigma)+\int \psi \,\mathrm{d} \eta^{\prime}.
\end{aligned}
$$
This shows that  $\eta^{\prime}=\mu$, where $\mu$ is the unique Gibbs equilibrium measure for $\psi$ (see Section \ref{sec: Transfer}). Moreover, by the upper semi-continuity of the Lyapunov exponents,
$$
\begin{aligned}
\lambda_1\left(\mu, \mathcal{A}\right)+\lambda_2\left(\mu, \mathcal{A}\right) & \geq \limsup _{n \rightarrow \infty} \lambda_1\left(\eta_{t_n}, \mathcal{A}\right)+\lambda_2\left( \eta_{t_n}, \mathcal{A}\right) \\
& =\limsup_{n \rightarrow \infty} 2 \lambda_1\left(\eta_{t_n}, \mathcal{A}\right) \\
& = \limsup_{n \rightarrow \infty} \Big[ 2 \int \psi\, d\eta_{t_n} + {2t_n}  \lim_{k\to\infty} \frac1k \int  \log \|\A^k(x)\|\, d\eta_{t_n}(x)\Big] 
\\
&= 2 \lambda_1\left(\mu,\A\right)
\end{aligned}
$$
This contradicts the fact that the cocycles have simple Lyapunov spectrum (cf. \cite{bonatti2004lyapunov}).  This proves the claim. 
\end{proof}

\medskip
The claim implies the existence of the Oseledets’ subspace $\xi_{\ast} (x)$ with the slowest Lyapunov exponent.

We are now in a position to show that the Gibbs measures ${\hat \mu_t}$ are the unique equilibrium states for each $t\in (-t_3,0)$. The argument is to prove that every such equilibrium state is a Gibbs measure, and to use the uniqueness of Gibbs measures.
Let $\tilde{\mu}$ be an ergodic equilibrium measure for $\Phi_{\A, t}$. By Claim~1, 
$\lambda_1(\tilde{\mu}, \mathcal{A})>\lambda_2(\tilde{\mu}, \mathcal{A})$. 
We will show that $\tilde{\mu}={\hat \mu_t}$ (this implies that ${\hat \mu_t}$ is the unique equilibrium state, since each equilibrium measure is a barycentre of the collection of all ergodic equilibrium states).
Define
$$ 
Y=\left\{x \in \Sigma^+: \overline{\xi_*}(x) \text { is well defined and 0-dimensional}\right\}
$$ 
and note that $\sigma^{-1}(Y)=Y$. Consider the potential $\phi_{\mathcal{A}}: \Sigma^+ \rightarrow \mathbb{R} \cup\{-\infty\}$ defined by
$$
\phi_{\mathcal{A}}(x):= \begin{cases}-\log \left\|\mathcal{A}_*^{-1}(x) \xi_*(x)\right\|, & x \in Y \\ -\infty, & x \notin Y .\end{cases}
$$
It is simple to check that
\begin{equation}\label{LE as Birkhoff avarges}
    \int \phi_{\mathcal{A}} \mathrm{d} \tilde{\mu}=\lambda_1(\tilde{\mu}, \A).
\end{equation}
We also denote 
$$
g_t(x):=\frac{g(x) e^{t \phi_{\mathcal{A}}(x)}}{\rho_t} \frac{{\hat h_t}\left(x, \overline{\xi_*}(x)\right)}{{\hat h_t}\left(\sigma x, \overline{\xi_*}(\sigma x)\right)}
$$
when $x \in Y$ and 0 otherwise, which can be shown to be a $g$-function. 
By \eqref{LE as Birkhoff avarges}, the definition f $g_t$ and equality $P\left(\sigma,\Phi_{\A, t}^g \right)=\log \rho_t$ (see \eqref{TP is equal to the log SR}), we have
$$
h_{\tilde{\mu}}(\sigma)+\int \log g_t \mathrm{~d} \tilde{\mu}=h_{\tilde{\mu}}(\sigma)+\int \log g \mathrm{~d} \tilde{\mu}+t \lambda_1(\tilde{\mu}, \A)-P\left(\sigma,\Phi_{\A, t}^g \right)=0,
$$
thus concluding that $\tilde{\mu}$ is an equilibrium state for $\log g_t$. Hence, by \cite{Led74} we have that
$ 
\L_{\log{g_t}}^* \tilde{\mu}=\tilde{\mu}.
$ 
We now use the fact that $\sigma$ is topologically mixing. 
Let $k \in \mathbb{N}$ be such that for all $I, J \in \Sigma^*$ there exists $K \in \Sigma_k$ such that $I K J$ is admissible. 
In consequence,
given $n\ge 1$, $I \in \Sigma_n$ and $x \in \Sigma$ there is at least a point $z \in \sigma^{-(n+k)}(x)$ with $z \in[I]$. Hence,
$$
\begin{aligned}
\L_{\log g_t}^{n+k} 1_{[I]}(x)= & \rho_t^{-(n+k)} \sum_{y \in \sigma^{-(n+k)} (x)} g^{(n+k)}(y)\left\|\mathcal{A}_*^{-(n+k)}(y) \xi_*(y)\right\|^{-t} \frac{{\hat h_t}\left(y, \overline{\xi_*}(y)\right)}{{\hat h_t}\left(x, \overline{\xi_*}(x)\right)} \\
\geq & \rho_t^{-(n+k)}\left(\inf {\hat h_t}\right)\left(\sup {\hat h_t}\right)^{-1} g^{(n+k)}(z)\left\|\mathcal{A}_*^{-(n+k)}(z) \xi_*(z)\right\|^{-t} \\
=& \rho_t^{-(n+k)}\left(\inf {\hat h_t}\right)\left(\sup {\hat h_t}\right)^{-1} g^{(n+k)}(z)\left\|\mathcal{A}^{[n+k]}(z) \xi_*(x)\right\|^t \\
\geq & \rho_t^{-(n+k)}\left(\inf {\hat h_t}\right)\left(\sup {\hat h_t}\right)^{-1} g^{(n+k)}(z)\left\|\mathcal{A}^{[n+k]}(z)\right\|^t \\
\geq & \left(\rho_t^{-k} e^{-k\|\log g\|_{\infty}}\left(\max_{y \in \Sigma^+}\left\|\mathcal{A}^k(y)\right\|\right)^t\left(\inf {\hat h_t}\right)\left(\sup {\hat h_t}\right)^{-1}\right). 
 \rho_t^{-n} g^{(n)}(z)\left\|\mathcal{A}^{[n]}(z)\right\|^t
\end{aligned}
$$
where the second equality uses 
\begin{equation}\label{relation betwwen the smallest and the biggest LE}
\left\|\mathcal{A}^{[n]}(x) \xi_*\left(\sigma^n x\right)\right\|=\left\|\mathcal{A}^{[n]}(x) \frac{\mathcal{A}_*^{-n}(x) \xi_*(x)}{\left\|\mathcal{A}_*^{-n}(x) \xi_*(x)\right\|}\right\|=\left\|\mathcal{A}_*^{-n}(x) \xi_*(x)\right\|^{-1}. 
\end{equation}
 and the second and third inequalities use that $t$ is negative.
Hence, by bounded distortion,
$$
\begin{aligned}
\L_{\log g_t}^{n+k} 1_{[I]}(x)  
{\geq}  & C^t \cdot \left(\rho_t^{-k} e^{-k\|\log g\|_{\infty}}\left(\max _{y \in \Sigma}\left\|\mathcal{A}^k(y)\right\|\right)^t\left(\inf {\hat h_t}\right)\left(\sup {\hat h_t}\right)^{-1}\right) \\
& \cdot e^{-n P\left(\sigma,\Phi_{\A, t}^g\right)} \inf_{z\in [I]}g^{(n)}(z)\left\|\mathcal{A}^n(I)\right\|^t
\end{aligned}
$$
for every $x\in \Sigma^+$.
Altogether this shows that there exists $C,C'>0$ (depending on $t$) such that for all $n\ge 1$ and $I \in \Sigma_n$,
$$
\tilde{\mu}([I])
= \tilde{\mu} (\L_{\log g_t}^{n+k} 1_{[I]}) \geq C e^{-n P\left(\sigma,\Phi_{\A, t}^g \right)}\left\|\mathcal{A}^n(I)\right\|^t \inf _{z \in[I]} g^{(n)}(z) \ge C' \hat \mu_t([I]).
$$
Thus, for every $n\ge 1$,
$$
\sum_{I\in \Sigma_n} \hat\mu_t([I]) \log \frac{\tilde \mu([I])}{\hat\mu_t([I])}
\ge \log C' > -\infty.
$$
Lemma 5.4 in \cite{Feng11} ensures that $\tilde{\mu}\ll {\hat \mu_t}$. Since both measures are $\sigma$-invariant and ergodic then ${\hat \mu_t}=\tilde{\mu}$. 
This completes the proof of the theorem. 
\end{proof}


\begin{proof}[Proof of Theorem \ref{thm:Main1}]
  The statements in the theorem are a direct consequence of Theorem~\ref{The measure is Gibbs} and Theorem~\ref{uniqness}, noticing that equilibrium states for $t\Phi_\A$ correspond to equilibrium states for the potential 
  $\Phi_{\A, t}$ in the special case that $\psi$ is the zero potential. 
\end{proof}

\section{Proof of Theorem \ref{thm:Main2}}\label{Sec: psi mixing}
In this section, we prove that the Gibbs measures for 1-typical cocycles have $\psi$-mixing.
Let us first recall some
necessary concepts. For any $\mathcal{A}:\Sigma^+ \rightarrow GL(d, \R)$ and $I\in \Sigma$, recall that 
\begin{equation*}\label{definition of the cocycle for words}
     \|\mathcal{A}(I)\|=\max_{x\in [I]} \|\mathcal{A}^{|I|}(x)\|.
\end{equation*}

We say that a fiber-bunched linear cocycle $\mathcal{A}$  is \textit{k-quasi-multiplicative} if there exist $c>0$ and $k \in \N$ such that for all $I, J \in \mathcal{L}$, there is $K=K(I, J) \in \Sigma_{k}$ such that $IKJ \in \Sigma^*$ and 
\begin{equation}
\label{eq:defquasimult}
   \|\mathcal{A} (IKJ)\|\geq c \|\mathcal{A}(I)\| \|\mathcal{A}(J)\|.
\end{equation}
\begin{lem}\label{k-QM}
Let $\A: \Sigma \to \glr$ be a 1-typical fiber-bunched cocycle. Then the cocycle $\A$ is $k$-quasi-multiplicative.
\end{lem}
\begin{proof}
This statement is similar to the proof of Theorem A in \cite{park2020quasi}, where it is shown that
if $\mathcal{A}$ is a 1-typical cocycle then there exist $c>0$ and $k \in \N$ such that for all $I, J \in \Sigma^*$, there is $K=K(I, J) \in \Sigma^*$ with $|K|\le k$ such that $IKJ \in \Sigma^*$ and 
\begin{equation}
\label{eq:defquasimultweakP}
   \|\mathcal{A} (IKJ)\|\geq c \|\mathcal{A}(I)\| \|\mathcal{A}(J)\|.
\end{equation}
Hence, in order to prove the theorem one needs to show that 
~\eqref{eq:defquasimultweakP} holds with words $K$ of constant length $k$. 
We start by noticing that, 
if $\ell_0\ge 1$ is the integer given by
\cite[Lemma~4.13]{park2020quasi}, for any $\ell \ge \ell_0$
the transition word $K=K(I,J)$ can be chosen such that 
\begin{equation}
\label{eqqK}
    |K(I,J)|= 2m + 2\bar \tau + n +\hat n + 2\ell
\end{equation}
where $m=m(I,J)$, $n=n(I,J)$ and 
$\hat n=\hat n(I,J)$ are constants determined by invariance of cones
(see Lemmas 4.7 and 4.12 in \cite{park2020quasi}) and
 $\bar\tau$ is a constant given by primitivity of the subshift of finite type, 
 all of them independent of $\ell$. 
In particular, there exists $C\ge 1$ so that $|K(I,J)|\le C + 2\ell$
(cf. \cite[page~1983]{park2020quasi}). 

\medskip
We now show that the length of the connecting words $K\in \mathcal{L}$ above can be chosen uniform  
we use the same notations as in \cite{park2020quasi} whenever possible, for the reader's convenience. 
Fix $k_0=C+3\ell_0$. Given $I,J \in \mathcal L$ choose 
$$
\ell(I,J)=\frac12[C+3\ell_0-2m(I,J) - 2\bar \tau - n(I,J) -\hat n(I,J)]
$$ 
which, by construction, satisfies 
$\ell(I,J) \ge \frac32 \ell_0$.
The argument described above guarantees that 
there exists $K=K(I,J)\in \Sigma^*$
satisfying equations ~\eqref{eq:defquasimultweakP} and \eqref{eqqK} with $\ell=\ell(I,J)$, hence $|K(I,J)|=k_0$.
This completes the proof of the lemma.

\end{proof}



\begin{proof}[Proof of Theorem~\ref{thm:Main2}]
The proof is inspired by \cite{MP-uniform-qm}. Let $t\geq 0$. By Theorem \ref{thm:Main1}, there is a unique Gibbs equilibrium measure  $\hat{\mu}_t$ for $t\Phi_{\A}$.
Since $\hat{\mu}_t$ is the Gibbs equilibrium measure for $t\Phi_{\A}$,
there exists $C_0>0$ such that for any $n \in \N$ and $I \in \Sigma_n$,
\begin{equation}\label{Gibbs-property}
C_{0}^{-1} \|\A^{n}(x)\|^{t} \leq e^{|I| P\left(\sigma,t\Phi_{\A}\right)} \hat{\mu}_t([I]) \leq C_{0} \|\A^{n}(x)\|^{t}
\end{equation}
for every $x  \in [I]$. Since the cocycle is 1-typical, by Lemma \ref{k-QM},  there exist an integer $m\in \N$ and constant $C_{1}>0$ such that for all $I, J \in \Sigma$ there exists $K \in \Sigma_m$ such that
\begin{equation}\label{k-qm-property}
 \|\A(IKJ)\| \geq C_1 \|\A(I)\| \|\A(J)\|.
\end{equation}

 Therefore, by the bounded distortion \eqref{BDD-dis}, \eqref{Gibbs-property} and \eqref{k-qm-property}, for every $I, J \in \Sigma^{*}$ we have
$$\begin{aligned}
C_1 \hat{\mu}_t([I]) \hat{\mu}_t([J]) &\leq C^2 C_{0}^{2} C_{1} e^{-(|I|+|J|) P\left(\sigma,t\Phi_{\A}\right)}  \|\A(I)\|^{t} \|\A(J)\|^{t} \\
&\leq C^2 C_{0}^{2} e^{-(|I|+|J|) P\left(\sigma,t\Phi_{\A}\right)}   \|\A(IKJ)\|^{t}  \\
&\leq C^2 C_{0}^{3} e^{|K| P\left(\sigma,t\Phi_{\A}\right)}  \hat{\mu}_t([IKJ]) \\
&\leq C^2 C_{0}^{3} e^{m P\left(\sigma,t\Phi_{\A}\right)} \sum_{|K|=m} \hat{\mu}_t([IKJ]) \\
& = C^2 C_{0}^{3} e^{m P\left(\sigma,t\Phi_{\A}\right)} \hat{\mu}_t\left([I] \cap \sigma^{-m-|I|}[J]\right)
\end{aligned}
$$
so that
\begin{equation}\label{one-side-thm1}
\hat{\mu}_t\left([I] \cap \sigma^{-m-|I|}[J]\right) \geq \kappa \hat{\mu}_t([I]) \hat{\mu}_t([J])
\end{equation}
where $\kappa:= C^2 C_{0}^{-3} C_{1} e^{-m P\left(\sigma,t\Phi_{\A}\right)}$. 

By \eqref{one-side-thm1}, for any $n \geq m$ we have that
$$
\begin{aligned}
\hat{\mu}_t\left([I] \cap \sigma^{-n-|I|}[J]\right) &=\sum_{|K'|=n-m} \hat{\mu}_t\left([IK'] \cap \sigma^{-m-|K'|-|I|}[ J]\right) \\
& \geq \kappa \sum_{|K'|=n-m} \hat{\mu}_t([IK']) \hat{\mu}_t([J]) \\
&=\kappa\hat{\mu}_t([J]) \sum_{|K'|=n-m} \hat{\mu}_t([IK']) \\
&=\kappa \hat{\mu}_t([I]) \hat{\mu}_t([J]) .
\end{aligned}
$$
Thus, we have, by an approximation argument, that
\[\liminf _{n \rightarrow \infty} \hat{\mu}_t\left(X \cap \sigma^{-n} Y\right) \geq \kappa^{-1} \hat{\mu}_t(X) \hat{\mu}_t(Y)\]
for all $X, Y$ Borel measurable. The above inequality implies that $\hat{\mu}_t$ is totally ergodic. On the other hand,
$$\begin{aligned}
\hat{\mu}_t\left([I] \cap \sigma^{-m-|I|}[J]\right)&=\sum_{|K|=m} \hat{\mu}_t([IKJ ])\\
& \leq C_{0} C \sum_{|K|=m} e^{-(|I|+|K|+|J|) P\left(\sigma,t\Phi_{\A}\right)} \|\A(IKJ)\|^{t}\\
& \leq C_{0} e^{-(|I|+|J|) P\left(\sigma,t\Phi_{\A}\right)} \|\A(I)\|^{t} \|\A(J)\|^{t} \left(\sum_{|K|=m} e^{-|K| P\left(\sigma,t\Phi_{\A}\right)} \|\A(K)\|^{t}\right)\\
& \leq C_{0}^{4} \hat{\mu}_t([I]) \hat{\mu}_t([J])\left(\sum_{|K|=m} \hat{\mu}_t([K])\right)\\
& =C_{0}^{4} \hat{\mu}_t([I]) \hat{\mu}_t([J])
\end{aligned}
$$
so that
\begin{equation}\label{other-side-thm1}
\hat{\mu}_t\left([I] \cap \sigma^{-m-|I|}[I]\right) \leq \delta \hat{\mu}_t([I]) \hat{\mu}_t([J]),
\end{equation}
where $\delta:=C_{0}^{4}.$

 The above inequality implies that $\hat{\mu}_t$ is mixing by \cite[Theorem 2.1]{Ornstein}. By an approximation argument, we have that

$$
\begin{aligned}
& \psi_n^*=\sup \left\{\frac{\hat{\mu}_t(A \cap B)}{\hat{\mu}_t(A) \hat{\mu}_t(B)}: A \in \bigvee_{i=n}^{\infty} \sigma^{-i} \mathcal{U}, B \in \bigvee_{i=-\infty}^{-1} \sigma^{-i} \mathcal{U}, \hat{\mu}_t(A) \hat{\mu}_t(B)>0\right\} \leq \delta\\
& \psi_n^{\prime}=\inf \left\{\frac{\hat{\mu}_t(A \cap B)}{\hat{\mu}_t(A) \hat{\mu}_t(B)}: A \in \bigvee_{i=n}^{\infty} \sigma^{-i}\mathcal{U}, B \in \bigvee_{i=-\infty}^{-1} \sigma^{-i} \mathcal{U}, \hat{\mu}_t(A) \hat{\mu}_t(B)>0\right\} \geq \kappa^{-1}
\end{aligned}
$$
for all $n \geq m$, where $\mathcal{U}$ defines in \eqref{definition of U}.
By \cite[Theorem 1]{bradley1983mixing}, this implies that $\hat{\mu}_t$ is $\psi$-mixing.  Also, it is easy to see that $\psi$-mixing implies weak Bernoulli. 

 Let $-t_{\ast}<t<0$.  The proof is similar to $t\geq 0$ case, but we include it here for the reader's convenience.

By Theorem \ref{thm:Main1}, for all $t\in (-t_{\ast}, \infty)$, there is a unique Gibbs equilibrium measure $\hat{\mu}_t$ for $t\Phi_{\A}$.  Since $\hat{\mu}_t$ is the Gibbs measure for $t\Phi_{\A}$,
there exists $C_0>0$ such that for any $n \in \N$ and $I \in \Sigma_n$,
\begin{equation}\label{Gibbs-property1}
C_{0}^{-1} \|\A^{n}(x)\|^{t} \leq e^{|I| P\left(\sigma,t\Phi_{\A}\right)} \hat{\mu}_t([I]) \leq C_{0} \|\A^{n}(x)\|^{t}
\end{equation}
for every $x  \in [I]$.  Since the cocycle is 1-typical, by Lemma \ref{k-QM},  there exist an integer $m\in \N$ and constant $C_{1}>0$ such that for all $I, J \in \Sigma$ there exists $K \in \Sigma_m$ such that
\begin{equation}\label{k-qm-property1}
 \|\A(IKJ)\|^t \leq C_1 \|\A(I)\|^t \|\A(J)\|^t.
\end{equation}
Also, by the super-multiplicative activity property, there is $C_2>0$ such that
\begin{equation}\label{super-multiplicative}
    \|\A(IKJ)\|^t \geq C_2 \|\A(I)\|^t \|\A(J)\|^t.
\end{equation}
 Therefore, by the bounded distortion \eqref{BDD-dis}, \eqref{Gibbs-property1} and \eqref{super-multiplicative}, for every $I, J \in \Sigma^{\ast}$ we have
$$\begin{aligned}
C_2 \hat{\mu}_t([I]) \hat{\mu}_t([J]) &\leq C^2 C_{0}^{2} C_{2} e^{-(|I|+|J|) P\left(\sigma,t\Phi_{\A}\right)}  \|\A(I)\|^{t} \|\A(J)\|^{t} \\
&\leq C^2 C_{0}^{2} e^{-(|I|+|J|) P\left(\sigma,t\Phi_{\A}\right)}   \|\A(IKJ)\|^{t}  \\
&\leq C^2 C_{0}^{3} e^{|K| P\left(\sigma,t\Phi_{\A}\right)}  \hat{\mu}_t([IKJ]) \\
&\leq C^2 C_{0}^{3} e^{m P\left(\sigma,t\Phi_{\A}\right)} \sum_{|K|=m} \hat{\mu}_t([IKJ]) \\
& = C^2 C_{0}^{3} e^{m P\left(\sigma,t\Phi_{\A}\right)} \hat{\mu}_t\left([I] \cap \sigma^{-m-|I|}[J]\right)
\end{aligned}
$$
so that
\begin{equation}\label{one-side-thm11}
\hat{\mu}_t\left([I] \cap \sigma^{-m-|I|}[J]\right) \geq \kappa \hat{\mu}_t([I]) \hat{\mu}_t([J])
\end{equation}
where $\kappa:= C^2 C_{0}^{-3} C_{1} e^{-m P\left(\sigma,t\Phi_{\A}\right)}$. 

By \eqref{one-side-thm11}, for any $n \geq m$ we have that
$$
\begin{aligned}
\hat{\mu}_t\left([I] \cap \sigma^{-n-|I|}[J]\right) &=\sum_{|K'|=n-m} \hat{\mu}_t\left([IK'] \cap \sigma^{-m-|K'|-|I|}[ J]\right) \\
& \geq \kappa \sum_{|K'|=n-m} \hat{\mu}_t([IK']) \hat{\mu}_t([J]) \\
&=\kappa\hat{\mu}_t([J]) \sum_{|K'|=n-m} \hat{\mu}_t([IK']) \\
&=\kappa \hat{\mu}_t([I]) \hat{\mu}_t([J]) .
\end{aligned}
$$
Thus, we have, by an approximation argument that
\[\liminf _{n \rightarrow \infty} \hat{\mu}_t\left(X \cap \sigma^{-n} Y\right) \geq \kappa^{-1} \hat{\mu}_t(X) \hat{\mu}_t(Y)\]
for all $X, Y$ Borel measurable. The above inequality implies that $\hat{\mu}_t$ is totally ergodic. On the other hand,
$$\begin{aligned}
\hat{\mu}_t\left([I] \cap \sigma^{-m-|I|}[J]\right)&=\sum_{|K|=m} \hat{\mu}_t([IKJ ])\\
& \leq C_{0} C_{1} C \sum_{|K|=m} e^{-(|I|+|K|+|J|) P\left(\sigma,t\Phi_{\A}\right)} \|\A(IKJ)\|^{t}\\
&  \stackrel{~\eqref{k-qm-property1}}{\leq} C_{0}C_{1} e^{-(|I|+|J|) P\left(\sigma,t\Phi_{\A}\right)} \|\A(I)\|^{t} \|\A(J)\|^{t} \left(\sum_{|K|=m} e^{-|K| P\left(\sigma,t\Phi_{\A}\right)} \|\A(K)\|^{t}\right)\\
& \leq C_{0}^{4} C_{1}\hat{\mu}_t([I]) \hat{\mu}_t([J])\left(\sum_{|K|=m} \hat{\mu}_t([K])\right)\\
& =C_{0}^{4} C_{1}\hat{\mu}_t([I]) \hat{\mu}_t([J])
\end{aligned}
$$
so that
\begin{equation}\label{other-side-thm1b}
\hat{\mu}_t\left([I] \cap \sigma^{-m-|I|}[I]\right) \leq \delta \hat{\mu}_t([I]) \hat{\mu}_t([J]),
\end{equation}
where $\delta:=C_{1} C_{0}^{4}.$
The rest of the proof is similar to the case of positive $t$.
Finally, the quasi-Bernoulli property follows from equations
\eqref{one-side-thm1} and \eqref{other-side-thm1} for positive $t$ and equations
~\eqref{one-side-thm11} and ~\eqref{other-side-thm1b} for negative $t$.    
This completes the proof of the theorem.
\end{proof}

We derive the following consequence.

\begin{cor}
    \label{mut have LPS}
There is $t_{\ast}>0$ such that for each $t\in (-t_{\ast},t_{\ast})$:
    \begin{enumerate}
        \item The probability measures $\hat{\mu}_t$ and $\pi_*^+ \hat{\nu}_t$ have local product structure; 
        \item 
         For every $x \in \Sigma^{+}$, we have
\begin{equation}\label{conditional measures}
\hat{m}_{t,x}=\int m_{t,y} \, d \hat{\mu}_{t,x}^s(y)=\int \delta_{\xi(y)} \; d \hat{\mu}_{t,x}^s(y),
\end{equation}
where $(\hat{\mu}_{t,x}^s)_{x\in \Sigma^+}$ is a disintegration of $\hat{\mu}_t$ on the partition $\{W^s_{\text{loc}}(x)\colon x\in \Sigma^+\}$;
    \end{enumerate}
\end{cor}

\begin{proof}
By Theorem~\ref{thm:Main1}, there is $t_{\ast}>0$ such that for each $t\in (-t_{\ast},t_{\ast})$, $\hat{\mu}_t=\pi_*^+\hat{m}_t$
 is a Gibbs measure.


By Theorem~\ref{thm:Main2}, $\hat{\mu}_t$ satisfies \eqref{eq:LPSm}, so $\hat{\mu}_t$ has local product structure. 
The statment for $\pi_*^+\hat{\nu}_t$ follows from the fact that these measures are absolutely continuous with respect to each other with density bounded away from zero and infinity. 
Finally, 
item (2) follows as a direct consequence of item (1), together with 
Theorem~\ref{BV's results}.
\end{proof}

\section{Hyperbolic repellers and Anosov diffeomorphisms}\label{proof of TheoremC}
In this section, we prove Theorem~\ref{main:t2} and Theorem~\ref{main:t1}.
The strategy used in the proof of both results is to derive a relation between hyperbolic dynamical systems and linear cocycles.

Let us introduce some terminology in case of hyperbolic repellers.
We say that a $\theta$-fiber-bunched repeller $\Lambda$ associated to a $C^{1+\theta}$-map $T$ is \emph{1-typical} if the following conditions are satisfied:
\begin{itemize}
\item[1)] There exists such a periodic point $p_0 \in \Lambda$ such that the eigenvalues of the matrix $A(p_0):=DT^{\operatorname{per}\left(p_0\right)}{p_0}$ 
have multiplicity $1$ and distinct absolute value;
\item[2)] There exists a sequence of points $\left\{z_n\right\}_{n\ge 1_0} \subset \Lambda$ such that
$$
z_0=p_0, \; T (z_n)=z_{n-1} \text {, and } z_n \xrightarrow{n \rightarrow \infty} p_0
$$
and so that, for each $1\le t \le d$, 
the eigenvectors $\left\{v_{1}, \ldots, v_{d}\right\}$ of $A(p_0)$
are such that, for any $I, J \subset \{1, \ldots, d\}$ with $|I|+$ $|J| \leq d$, the set of vectors
$$
\left\{ \tilde{H}_{p_0}^{\left\{z_n\right\},-} \left(v_{i}\right): i \in I\right\} \cup\left\{v_{j,t}: j \in J\right\}
$$
is linearly independent, where  
\begin{equation}
\label{defholnoninv}
\tilde{H}_{p_0}^{\left\{z_n\right\},-}:=\lim _{n \rightarrow \infty}\left(DT{(p_0)} \right)^n\left(DT(z_n)  \right)^{-1} \ldots\left(DT({z_1}) \right)^{-1}.
\end{equation}
\end{itemize}

As in the invertible setting, in case the periodic point $p_0$ given by the previous definition is not fixed, we can consider a power of the $C^{1+\theta}$-map. In this way we will always assume that $p_0$ is a fixed point of $T$ (recall Remark \ref{fixed point}).

\begin{rem}
 It is worth noticing that the choice of a pre-orbit of $p_0$ in item (2) replaces the homoclinic loop in case of invertible maps and that the existence of the limit in ~\eqref{defholnoninv}
is guaranteed by  $C^{\theta}$-regularity of $D T$  and the $\theta$-bunching assumption on  $\Lambda$. Indeed, given $p\in \Sigma$ so that $\pi_1(p)=p_0$, an homoclinic point $z$, the identification $L(\pi_1 p): \mathbb{R}^d \rightarrow T_{p_0} M$ and the cocycle 
$\mathcal{C}$ over $\left(\Sigma, \sigma^{-1}\right)$
(cf. \eqref{def:cocliminv} below), the canonical holonomy $H_{ p \leftarrow z}^{s,-}$ is given by
$$
\begin{aligned}
H_{z, p}^{s,-}= & \lim _{n \rightarrow \infty} \mathcal{C}^n(p)^{-1} \mathcal{C}^n(z)=\lim_{n \rightarrow \infty} L(\pi_1 p)^{-1} \\
& {\left[\left((D_{p_0} h)\right)^n\left((D_{z_n} h) \right)^{-1} \ldots\left((D_{z_1} h)\right)^{-1}\right] L(\pi_1 z) }
\end{aligned}
$$
and, since $L(\pi_1 z)=L(\pi_1 p)$, $H_{z \leftarrow p}^{u,-}=\text{Id}$. In particular
the holonomy loop for $\mathcal C$, defined by $\tilde{H}_{p \leftarrow z}^{s,-} :=H_{z, p}^{s,-} \circ H_{p, z}^{u,-}$, relates to 
$\tilde{H}_{p_0}^{\left\{z_n\right\},-}$
in ~\eqref{defholnoninv} by the conjugacy relation
\begin{equation}\label{relation betwen loops}
   \tilde{H}_{p}^{\left\{z\right\},-}=L(\pi_1 p)^{-1} \circ \tilde{H}_{p_0}^{\left\{z_n\right\},-} \circ L(\pi_1 p) .
\end{equation}
\end{rem}

We will need the following:

\begin{lem}\label{dense subset}
    Let $M$ be a Riemannian manifold, and let $T: M \rightarrow M$ be a $C^r$ map with $r>1$. Assume $\Lambda \subset M$ is a $\theta$-bunched repeller defined by $T$ for some $\theta \in(0,1)$ satisfying $r-1>\theta$. There exists a $C^1$-open neighborhood $\mathcal{V}_1$ of $T$ in $C^r(M, M)$ and a $C^1$-open and $C^r$-dense subset $\mathcal{V}_2$ of $\mathcal{V}_1$ such that $\Lambda_S$ is 1-typical for every $S \in \mathcal{V}_2$.
\end{lem}


 \begin{proof}
 Let $T\in C^r(M,M)$ be as above.
  It follows from the proof of \cite[Lemma 5.10]{park2020quasi} that there exists a $C^1$-open neighborhood $\mathcal{V}_1$ of $T$ in $C^r(M, M)$ and a $C^1$-open and $C^r$-dense subset $\mathcal{V}_2$ of $\mathcal{V}_1$ such that $\Lambda_S$ is 1-typical for every $S \in \mathcal{V}_2$.

 \end{proof}

We are now in a position to complete the proof of Theorem \ref{main:t2}. 
\begin{proof}[Proof of Theorem \ref{main:t2}]

Let us first recall some general facts. As $\Lambda$ is a repeller, there exists a finite Markov partition $\mathcal{R}$ for $\Lambda$, a one-sided subshift of finite type $(\Sigma^+,\sigma^+)$, $\Sigma^+\subset \{1,2,\dots, q\}^\N$, and a H\"older continuous and surjective coding map
\begin{equation}
\label{defxi1}
\chi: \Sigma^{+} \rightarrow \Lambda_S    
\end{equation}
such that $\chi \circ \sigma^{+}=T \circ \chi$.

By fixing a Markov partition of sufficiently small diameter, we may ensure that the $\chi$-image of each cylinder $[j]$ of $\Sigma^{+}, 1 \leq j \leq q$, is contained in an open set on which $T M$ is trivializable. 

\medskip
Consider the natural extension $\left(\Sigma, \sigma\right)$ of $\left(\Sigma^{+}, \sigma^{+}\right)$, and its inverse $\left(\Sigma, \sigma^{-1}\right.)$.
Let $\Pi: \Sigma \rightarrow \Sigma^{+}$ denote the natural projection.
For each $1 \leq j \leq q$ and $y \in[j] \subset \Sigma^{+}$, let $L(y):=L_j(y): \mathbb{R}^d \rightarrow T_{\chi(y)} M$ be a fixed trivialization of $T M$ over an open neighborhood containing $\chi[j]$.  We define a cocycle $\mathcal{C}$ over $\left(\Sigma, \sigma^{-1}\right)$ by
\begin{equation}
\label{def:cocliminv}
\mathcal{C}(x):=L\left(\Pi \sigma^{-1} x\right)^{-1} \circ D_{\chi\left(\Pi \sigma^{-1} (x)\right)}T)^{-1} \circ L(\Pi (x)),
\end{equation}
which can be thought of as the inverse of the 
derivative cocycle $D T\mid_\Lambda$
over $\left(\Sigma, \sigma\right)$ defined in the obvious way.
For any $n\ge 1$, we have
$$
\mathcal{C}^n\left(\sigma^n (x)\right)=L(\Pi x)^{-1} DT^n({\chi(\Pi x)})^{-1} L\left(\Pi \sigma^n (x)\right).
$$

We proceed to reduce the proof of the theorem to the invertible setting.  
By the proof of \cite[Lemma 5.8]{park2020quasi}, for any $\mu \in \M\left(\sigma\right)$ and $\nu \in \mathcal{M}_{\text{inv}}(T)$ related by $\chi_* \mu=\nu$, we have
\begin{equation}\label{entropies and LE are equall}
    h_\mu\left(\sigma\right)=h_\nu(T)
\quad \text{and}\quad 
\lambda_1(\mu, \mathcal{C})= \lambda_1(\nu, DT).
\end{equation}

Now, given $S\in \mathcal V_2$ the repeller $\Lambda_S$ is 1-typical (see Lemma \ref{dense subset}) and, consequently, there cocycle $\mathcal{C}_S$ defined by \eqref{def:cocliminv} is a 1-typical cocycle.  
Therefore, using the latter and \eqref{varitional}, we conclude that
\begin{equation}\label{pressures are equal}
    P(S, t\log \|DS_{|\Lambda_{S}}\|)=P(\sigma, t\log \|\mathcal{C}_S\|).
\end{equation}

Let $T \in C^r(M, M)$ be as in the statement of the theorem and let $S$ be a $C^1$-small perturbation of $T$. 
If the perturbation is sufficiently small, then 
one can use the same  trivialization over $T_{\Lambda} M$ to code the dynamics of $S$ on $\Lambda_S$ via a conjugacy $\chi_S$ (cf. equation~\eqref{defxi1}), and take its natural extension. Then we realize the perturbation $\left.T\right|_{\Lambda}$ to $\left.S\right|_{\Lambda_S}$ as the perturbation of the cocycle $\mathcal{C}$ to $\mathcal{C}_S$ over the same subshift of finite type $\left(\Sigma, \sigma^{-1}\right)$. 
In this way, Theorem~\ref{main:t2} is a consequence of \eqref{pressures are equal}, \eqref{entropies and LE are equall}, Theorems~\ref{thm:Main1} and \ref{thm:Main2}.
Moreover,  Friedman and Ornstein's result \cite{Friedman70} guarantees that the weak Bernoullicity implies that it is conjugate to a Bernoulli shift.

\end{proof}

Now, we prove Theorem \ref{main:t1}, on the non-additive thermodynamic formalism of Anosov diffeomorphisms. 


\begin{proof}[Proof of Theorem \ref{main:t1}]

Assume that $T$ is an Anosov diffeomorphism.
Denoting the dimension of the unstable bundle $E^u$ by $d$, one can realize $\left.D T\right|_{E^u}$ as a 
$\text{GL}(d,\R)$-cocycle over a suitable subshift of finite type $\left(\Sigma, \sigma\right)$. Indeed, the existence of a finite Markov partition for $T$ \cite{Bow} results in a H\"older continuous surjection $\pi_1: \Sigma \rightarrow M$ such that $T \circ \pi_1=\pi_1 \circ \sigma$. By choosing a Markov partition of sufficiently small diameter, one  may assume that the image of each cylinder $[j]$ of $\Sigma, 1 \leq j \leq q$, is contained in an open set on which $E^u$ is trivializable. For $x \in[j]$, we let $L_j(x)$ : $\mathbb{R}^d \rightarrow E_{\pi_1 (x)}^u$ be a fixed trivialization of $E^u$ over $\pi_1([j])$. 
We define the $\alpha$-H\"older $\text{GL}(d,\R)$-cocycle 
$\mathcal{A}$ over the subshift $\left(\Sigma, \sigma\right)$ by
\begin{equation}\label{definition of cocycle for Ansosov}
    \mathcal{A}(x):=\left.L_k(\sigma (x))^{-1} \circ DT({\pi_1 (x)}) \right|_{E^u} \circ L_j(x),
    \text{whenever $\sigma (x) \in [k]$.}
\end{equation}

 The assumption on the periodic point $p$ guarantees that, defining  
the cocycle $\A : \Sigma \to GL(d,\mathbb R)$ by \eqref{definition of cocycle for Ansosov},
the cocycle $\A^n: \Sigma \to GL(d,\mathbb R)$ over the shift $(\Sigma,\sigma^n)$ is 1-typical. 
Then the proof follows from Theorems ~\ref{thm:Main1} and \ref{thm:Main2}.
This completes the proof of the theorem. 
\end{proof}

\subsection*{Acknowledgements.}
RM is supported by the Swedish Research Council grant 104651320. PV was partially supported by CIDMA under the Portuguese Foundation for Science and Technology (FCT, https://ror.org/00snfqn58)  
Multi-Annual Financing Program for R\&D Units.

\bibliographystyle{acm}
\bibliography{08july2025}

\end{document}

%% file: general-preamble.tex
\usepackage{geometry}
\usepackage{amsmath}

\numberwithin{equation}{section}
\numberwithin{figure}{section}
\usepackage{enumitem}		
 \let\footnote=\endnote
\@ifundefined{lettrine}{\usepackage{lettrine}}{}

\theoremstyle{definition}

\def\L{\Lambda}

\def\R{\mathbb{R}}

\def\N{\mathbb{N}}
\def\Z{\mathbb{Z}}

\def\ep{\varepsilon}